\documentclass[12pt]{amsart}

\usepackage{amsfonts,latexsym,rawfonts,amsmath,amssymb,amsthm}
\usepackage[plainpages=false]{hyperref}

\usepackage{mathrsfs}
\usepackage{enumerate}

\usepackage{ragged2e}

\usepackage{graphicx}

\RequirePackage{color}

\numberwithin{equation}{section}

\theoremstyle{plain}

\newtheorem{corollary}{Corollary}[section]
\newtheorem{definition}{Definition}[section]

\newtheorem{lemma}{Lemma}[section]

\newtheorem{theorem}{Theorem}[section]

\newcommand{\beq}{\begin{equation}}
\newcommand{\eeq}{\end{equation}}
\newcommand{\beqs}{\begin{eqnarray*}}
\newcommand{\eeqs}{\end{eqnarray*}}
\newcommand{\beqn}{\begin{eqnarray}}
\newcommand{\eeqn}{\end{eqnarray}}
\newcommand{\beqa}{\begin{array}}
\newcommand{\eeqa}{\end{array}}

\def\Om{\Omega}
\def\pom{\partial\Omega}
\def\bom{\overline\Omega}

\def\p{\partial}
\def\eps{\varepsilon}
\def\J{\mathcal J}
\def\I{\mathcal I}

\newcommand{\R}{\mathbb R}

\begin{document}

\title  [A new minimax principle]
{ {A new minimax principle and \\ 
application to the $p$-Laplace equation}}

 \author{Xu-Jia Wang \&  Xinyue Zhang}
\address{ Xu-Jia Wang, 
Institute for Theoretical Sciences, Westlake University, Hangzhou, 310030, China.}
\email{wangxujia@westlake.edu.cn}

\address{Xinyue Zhang, 
Institute for Theoretical Sciences, Westlake University, Hangzhou, 310030, China.}
\email{zhangxinyue@westlake.edu.cn}

%%% \author{Xinyue Zhang}
%%% \address{ }
%%% \email{ }



% \thanks{This work was supported by the Start-up fund of Westlake University. }

\maketitle

\begin{abstract} 

We introduce a new minimax principle to prove the existence of multi-peak solutions 
to the Neumann problem of the $p$-Laplace equation
$$
-\eps^p \Delta_p u =  u^{q-1}  -  u^{p-1} \ \ \text{in}\ \Om,
$$
where $\Om$ is a bounded domain in $\R^n$ with smooth boundary,
$1<p<n$ and $p<q< \frac{np}{n-p}$.
The minimax principle will be applied to the set of peak functions,
which is a subset of the Sobolev space $W^{1,p} (\Om)$. 
The argument is based on a combination of variational method, topological degree theory, and gradient flow.

\end{abstract}

\baselineskip16.2pt
\parskip5pt

\section{Introduction}

The minimax principle, such as the Mountain Pass Lemma, 
is one of the most useful tool in the critical point theory,
and has been extensively applied in the study of nonlinear elliptic equations \cite{R84, Will}.

In this paper, we introduce a new minimax principle, 
by combining the gradient flow and the topological degree theory.
We use the new minimax principle to prove the existence of single or multiple peak solutions
to the Neumann problems of semilinear and quasilinear elliptic equations.
Our proof avoids the well-known Liapunov-Schmidt reduction, 
and can be applied to more general variational problems.

We will study the Neumann problem  of the $p$-Laplace equation
\beq\label{NP}
{\begin{split}
-\eps^p \Delta_p u & =  u^{q-1} -  u^{p-1} \  \ \text{in}\ \Om, \\
u_\nu  & =0 \ \ \ \text{on}\ \pom , \\
u & >0\ \ \ \text{in}\ \ \Om,
\end{split}}
\eeq
where $\Om$ is a smooth, bounded domain in the Euclidean space $\R^n$,
$\nu$ is the unit outer normal, and $\eps>0$ is a small  constant.
We assume that $1<p<n$ and $p<q< \frac{np}{n-p}$.
As usual, we denote the $p$-Laplace operator by $\Delta_p$, which is given by
$$\Delta_p u = \text{div}(|Du|^{p-2} Du).$$

Single or multiple peak solutions of elliptic equations with Neumann boundary condition
have attracted significant interest 
due to the special geometric feature of the solutions, 
and  have been extensively studied since the 1980s.
This topic was first investigated by Lin, Ni and Tagaki \cite {LNT88, NT86, NT91, NT93}, 
who obtained the least energy solutions and proved that, for sufficiently small $\eps$, 
the least energy solution is a single peaked function, and the peak is located on the boundary near point where the mean curvature attains its maximum.  
After that this topic has received considerable attention and has led to numerous significant contributions. 
Let us recall several results concerning multi-peak solutions.
%the least energy solution is a single peaked function, and the peak is located on the boundary. 
%After that this topic attracted a lot of attention and a great deal of research was carried out. 
%Let us recall a few works on multi-peak solutions. 
In the papers \cite{CK99, Gui96, Wa95}, 
the authors established the existence of solutions with multiple boundary peaks under various conditions
on the boundary of the domain.
In the papers \cite{CW98, DY99, GW99}, interior multi-peak solutions were constructed,
%the authors proved the existence of interior multi-peaked solutions, 
while the authors in  \cite{DY99b, GW00, GWW} proved the existence of both interior and boundary 
multi-peaked solutions under certain conditions on the boundary. 
In general, boundary peaks are associated with the mean curvature of the boundary, 
whereas interior peaks are related to the distance function. 
An important tool in these studies is the Liapunov-Schmidt reduction method. 
This technique requires an estimate of the third eigenvalue of the linearized equation,
such that the linearized equation is locally invertible. 
A different approach, called the manifold approach, was utilized in \cite{Ko99}. 
The main idea of this method is to construct a family of approximate solutions and 
then show the existence of a true solution near this family.
This method again needs the estimate for the eigenvalue of the linearized equation.
As a result, all the above mentioned papers deal with the Laplace operator, namely the case when $p=2$.

There are many more works on the existence of nontrivial solutions to the Neumann problem in the case $p=2$,
such as solutions concentrating on a component of the boundary \cite{MM02, MM04},
or non-constant solutions which disprove the Lin-Ni's conjecture \cite{DMM14, DMRW, WWY1, WWY2}.
We also refer the reader to \cite{CPY, DY06, GL02,  GWY, LNW, Rey, RW05, Wa92} 
for different developments on this topic.

For the $p$-Laplace equation, estimate for the eigenvalue of the linearized $p$-Laplace operator is lacking, 
which remains unclear whether one can apply the Liapunov–Schmidt reduction method or the manifold approach, to the $p$-Laplace equation \eqref{NP}.
Therefore, no results on the existence of multi-peak solutions for the $p$-Laplace equation are currently available, as such analysis requires a completely different approach.

%For the $p$-Laplace equation, 
%estimate for the eigenvalue of the linearized $p$-Laplace equation is lacking,
%it is unclear whether one can apply the Liapunov-Schmidt reduction, or the  the manifold approach, 
%to the $p$-Laplace equation  \eqref{NP}.
%Therefore, no result on the existence of multi-peak solutions to the $p$-Laplace equation
%can be found, it requires a completely different approach.

%% The arguments also rely on the exponential decay of the positive radial solution to the equation
%% \beq\label{Kwong} -\Delta w+w=w^p\ \ \text{in}\ \R^n , \eeq
%% where $1<p<\frac{n+2}{n-2}$. 
%% It is known that there is a unique positive solution to \eqref{Kwong} which decays to 0 at infinity,
%% and the solution is rotationally symmetric.
%% It is unclear whether the arguments can be applied to
%% the $p$-Laplace equation \eqref{NP}, or elliptic equations with nonsmooth variable coefficients.

In this paper, we introduce a new method 
that utilizes the idea of the minimax principle, 
in combination with topological degree theory and gradient flow, 
to establish the existence of multi-peak solutions to 
the $p$-Laplace equation \eqref{NP}.

\begin{theorem}\label{T1.1}
Let $\Om$ be a smooth, bounded domain in $\R^n$.
Assume that $1<p<n$ and $p<q< \frac{np}{n-p}$.
Then for any  integers $k, l\ge 0$, $k+l\ge 1$, 
the Neumann problem \eqref{NP} admits a solution
with interior $k$ peaks and boundary $l$ peaks, as long as $\eps$ is sufficiently small.
\end{theorem}

For semilinear elliptic equations, namely the case when $p=2$,
our argument provides a different, but simpler proof for the existence of solutions with both interior and boundary peaks.
Furthermore, the minimax \eqref{S*} also gives rise to a new observation on the solution, 
namely it reaches the maximum of the functional $\I$  in the parameters {\bf a} and {\bf b}.
Note that we do not impose any extra conditions on the boundary $\pom$ for the existence 
of solutions with multi-peaks on the boundary. 
Our method can also be applied to study other variational problems, 
%{\color{blue} 
such as the nonlinear Schr\"odinger equation in $\R^n$, which we plan to treat in subsequent papers.
%}
%% {\color{blue}
%% In a subsequent paper we plan to apply the method to study boundary layers
%% for singularly perturbed Neumann problem \cite{MM02, MM04}. }

We point out that the solutions we obtained in Theorem \ref{T1.1}, in the case $k+l=1$,
are not  the least energy solution obtained in \cite{NT91, NT93}. 
The solutions we obtained have higher energy and are unstable solutions.
%% The technique of searching unstable solutions in this paper is inspired by   \cite{Wa96}.

Our proof needs the uniqueness of the ground state solution to \eqref{ST00}, 
proved for $1<p<n$ in \cite{PS98, ST00}.
That is, there is a unique ground state solution $w$ to 
\beq\label{ST00}
-\Delta_p u =  u^{q-1} -  u^{p-1} \  \ \text{in}\ \R^n,
\eeq
which is rotationally symmetric \cite{SZ99}.
Moreover, $w$ is the minimizer of the functional
\beq\label{Min1}
\J (u)= \frac{\int_{\R^n} (|Du|^p+|u|^p)dx}{\big( \int_{\R^n} |u|^q dx\big)^{p/q}}, \ \ \ u\in W^{1,p}(\R^n).
\eeq
When $p=2$, it was proved that $w$ decays exponentially \cite{GNN}. 
For $p\ne 2$, the exponential decay  has also been obtained in \cite{dMY05, LZ05b}.

This paper is arranged as follows. In Section \ref{S2}, we introduce peak functions.
In Section \ref{S3},  we construct the new minimax principle. 
To complete this construction, a gradient flow is introduced in Section  \ref{S4}. 
In Section \ref{S5}, we apply the gradient flow together with topological degree theory to prove Theorem \ref{T1.1}.

\vskip20pt

\section{Peak functions}\label{S2}

To prove Theorem \ref{T1.1}, 
we search for peak solutions such that, near each peak point $p$,
the solution looks like $w(\frac{x-p}{\eps})$  in the sense of \eqref{peak1},
where $w$ is the unique positive solution to \eqref{ST00} in $W^{1,p}(\R^n)$. 
For this purpose,  we first introduce the notion of peak functions. 
For clarity, this section is divided into several subsections.

\subsection{Peak function}\label{S2.1}

\vskip10pt
%\justifying
\begin{definition}\label{D2.1}
A nonnegative function $u: \Om \rightarrow \mathbb{R}$ is a peak function with interior peaks $p_1, \cdots,  p_k \in \Om $ 
and boundary peaks $q_1, \cdots, q_l\in\pom$  if the following conditions hold:
\begin{itemize}

\item [(i)] 
There exist positive constants $a_1, \cdots, a_k, b_1, \cdots, b_l\in (1-\hat \delta, 1+\hat\delta)$, 
such that 
\beq\label{peak1}
 \big{\|}u(x)- {\Small\text{$\sum_{i=1}^k$}} a_i w\big( {\Small\text{$\frac{x-p_i }{\eps}$}} \big)
            - {\Small\text{$\sum_{j=1}^l$}} b_j w\big(  {\Small\text{$\frac{x-q_j }{\eps}$}} \big)\big{\|}_{L^\infty(\Om)} < {\bar\delta} ,
\eeq
where $w$ is the unique ground state solution to \eqref{ST00}.

\item [(ii)] $p_1, \cdots, p_k $ and $q_1, \cdots, q_l$
are local mass centres of $u$.

\item [(iii)] $p_1, \cdots, p_k$ and $q_1, \cdots, q_l$
satisfy the ${\delta}$-apart condition for a small constant $\delta>0$.

\item [(iv)]   $u\in C^\gamma (\bom)\cap W^{1,p}(\Om)$ for some $\gamma\in (0, 1)$, and 
\beq\label{peak2}
\eps^{-n} \int_\Om (\eps^p\, |Du|^p+|u|^p) dx \le C,
\eeq
%\ \vskip-15pt \ 
\beq\label{peak3}
\|u\|_{C^\gamma(\bom)}\le C\eps^{-\gamma} ,
\eeq
where the constant $C$ depends only on $k, l, n$.

\end{itemize}

\end{definition}

Given a peak function $u$ as in Definition \ref{D2.1}, we will show below that 
the coefficients $a_{i}, b_j$ and the peaks $p_i$, $q_j$ can be uniquely determined. 
This property is essential for the topological degree theory in the following argument.

Throughout this paper, we use $\eps, \delta, \bar\delta, \hat\delta$ to denote small constants satisfying 
$0<\eps, \delta, \bar\delta, \hat\delta<\frac 14$. 
Here, $\eps$ always stands for the constant in equation \eqref{NP}.
$\bar\delta, \hat\delta$ in part (i) of Definition \ref{D2.1} are independent of $\eps$,
while $\delta$ in part (iii) may depend on $\eps$, 
such as $\delta=N\eps$ for a large constant $N>1$ ($N$ independent of $\eps$).
The constants $\delta, \bar\delta$ and $\hat \delta$ will be determined in the following argument.

For brevity, we denote ${\bf p}=(p_1, \cdots, p_k)$, ${\bf q}=(q_1, \cdots, q_l)$,  
${\bf a}= (a_1, \cdots, a_k)$ and ${\bf b}= (b_1, \cdots, b_l)$.
We define $\Phi_{k; l}({\bf p, q},  {\bf a, b})$ as the set of peak functions 
with interior peaks ${\bf p}\in\Om^k$ and boundary peaks ${\bf q}\in (\pom)^l$, 
and coefficients ${\bf a}$ and ${\bf b}$,
where $\Om^k=\Om\times\cdots\times\Om\subset \R^{n\times k}$, and\
$(\pom)^l=\pom\times\cdots\times\pom\subset\R^{n\times l}$.
We also denote
\beq\label{weps}
w_{{\bf p, q}, {\bf a, b}, \eps} (x) ={\Small\text{$ \sum_{i=1}^k$}} a_i w\big( {\Small\text{$\frac{x-p_i }{\eps}$}} \big)
            + {\Small\text{$\sum_{j=1}^l $}} b_j w\big(  {\Small\text{$\frac{x-q_j }{\eps}$}} \big)  .
\eeq

Therefore, we have
\beq\label{Phi}
\begin{aligned}
& \ \ \ \ \Phi_{k; l}({\bf p, q},  {\bf a, b}) \\
&  = \{u\in C^\gamma (\bom)\cap W^{1,p}(\Om)\ |\ u \ \mbox{is a peak function with} \ \|u-w_{{\bf p, q}, {\bf a, b}, \eps}\|_{L^\infty(\Om)}<\bar\delta \},
 \end{aligned}
\eeq
which is an open set in $C^\gamma (\bom)$ and a subset of  $W^{1,p}(\Om)$.

Let $u$ be a peak function with a peak  $p$, and 
let $y=\frac{x-p}{\eps}$. By \eqref{peak1} we have
\beq\label{peak1a}
|u(\eps y+p)- a w(y)| \le {\bar\delta} %\ \ \ \forall\ y\in B_{{\delta}/\eps}(0)\cap\Om^\eps ,
\eeq 
for any $y\in B_{{\delta}/\eps}(0)\cap\Om^\eps$,
where $\Om^\eps=\Om^\eps_p=\{y=\frac{x-p}{\eps}\ |\ x\in\Om\} $ (the subscript $p$ in $\Om^\eps$ is omitted for brevity).
Here, $B_r(p)$ denotes the ball of radius $r$ with centre at $p$ in $\R^n$ .
In the coordinates $y$, \eqref{peak2} becomes
\beq\label{peak2a}
 \int_{\Om^\eps} ( |Du|^p+|u|^p)dy \le C,
\eeq
and \eqref{peak3} implies 
\beq\label{peak3a}
\|u\|_{C^\gamma(\bom^\eps)} \le C.
\eeq
%% From the argument of the paper, we may take 
%% $$C=2(k+l)\int_{\R^n} (|Dw|^p+|u|^p).$$
Assumptions \eqref{peak2a} and \eqref{peak3a} ensure that all peak functions are uniformly bounded
in $W^{1,p}(\Om^\eps)$ and $C^\gamma(\bom^\eps)$.

If a peak function is  a solution to \eqref{NP}, one can deduce that $a_i, b_j\to 1$ as $\eps\to 0$, by applying a blow-up argument and the uniqueness of the ground state solution to \eqref{ST00}.  
Therefore, in Definition \ref{D2.1}, we assume that 
$a_i$ and $b_j$ are near 1. 

In Definition \ref{D2.1}, we assume that a peak function is H\"older continuous. 
By the regularity of the $p$-Laplace equation \cite{DB93}, 
we may assume that a peak function is $C^{1,\gamma}$, 
but we cannot assume it is $C^2$.

\subsection{${\delta}$-apart condition}\label{S2.2}

We say that the points $p_1, \cdots, p_k\in\Om$ and $q_1, \cdots, q_l\in\pom$
satisfy the ${\delta}$-apart condition if  the following properties hold:
\begin{itemize}
\item [$\diamond$]  \  $d_{p_i} > {\delta} $ for all $1\le i\le k$,  i.e.,  $p_i\in \Om_{{\delta}}$, 
\item [$\diamond$] \  $|p_i-p_j| > 2{\delta}$ for all $1\le i, j\le k$ and $i\ne j$, 
\item [$\diamond$]\   $|q_i-q_j| > 2{\delta}$ for all $1\le i, j\le l$ and $i\ne j$,
 \end{itemize}
where $p_1, \cdots, p_k$ are interior peaks and $q_1, \cdots, q_l$ are boundary peaks, 
$\Om_\delta=\{x\in\Om |\ d_x>\delta\}$ and $d_x=\text{dist}(x, \pom)$.

\vskip5pt

\subsection{Local mass centres of a peak function}\label{S2.4}

The points  $p_1, \cdots, p_k$ and $q_1, \cdots, q_l$
cannot be uniquely determined from \eqref{peak1} alone.
Indeed, if $p_1, \cdots, p_k$ satisfy  \eqref{peak1},
then by a very small translation of these points, they still satisfy \eqref{peak1}.
To ensure their uniqueness, %In order that they are uniquely determined, 
we introduce the notion of local mass centres and
we require that the peaks $p_i$ and $q_j$ are the local mass  centres of the function.

Let $u$ be a nonnegative measurable function in $\Om$, satisfying \eqref{peak1}.
We say that $p\in\Om $ is an {\it interior local mass  centre}  of $u$ if $d_p>{\delta}$, $|p-p_i|<\eps$ for some $1\le i\le k$, 
and  for all unit vector $e$ in $\R^n$, 
\beq\label{cen1}
\int_{\R^n} \chi_e (y)
          \, \eta(|y| \big)\, u(y+p) dy  =0.
\eeq
In \eqref{cen1}, the function $\chi_e$ is defined by
$$  \chi_e(y) =\bigg\{{\begin{split}
 & 1\ \ \ \ \ \ \  \ \text{if $\langle e, y \rangle>0$},\\[-2pt]
 & -1\ \ \  \ \text{if $\langle e, y \rangle<0$},
 \end{split}} $$
 where $\langle \cdot, \cdot \rangle$ denotes the inner product in $\R^n$.
 The function $\eta$ in \eqref{cen1} is a weight function, which we can choose simply as
\beq\label{eta}  
\eta(t)= \bigg\{ {\begin{split}
 & 1\ \ \ \ \text{if}\ |t|<{\eps}, \\ 
 & 0\ \ \ \ \text{if} \ |t|>{\eps} .
 \end{split}} \eeq
Alternatively, a different weight function may also be selected.
%%But we can also choose a different weight function.  
%% we can choose any function satisfying   
%% $\eta (0)=1$, $\eta (t)> 0$ for $t\in (-{\delta_1}, {\delta_1})$,  $\eta(t)=0$  when $|t|>{\delta_1}$, and
%% $\eta$ is monotone decreasing for $t\in (0, {\delta_1})$, for some constant $\delta_1\in (\eps, \delta)$.

Let $q\in\pom$ be a boundary point of $\Om$.
We say that $q$ is a {\it boundary local mass  centre}  of $u$ if $|q-q_j|<\eps$ for some $1\le j\le l$, 
and \eqref{cen1} holds for any unit vector $e$
in the tangential plane of $\pom$ at $q$.

Alteratively,  we can define a {\it boundary local mass centre} $q$ of $u$ on the boundary by first flattening 
the boundary in $B_{2\delta}(q)$, making an even extension of the function such that the boundary point 
becomes an interior point,
and then defining the boundary local mass centre exactly in the same way as the interior ones.

It is possible that there are many points in $\Om$ satisfying \eqref{cen1},
but we are only interested in points near $p_i \ (i=1, \cdots, k)$ and $q_j\  (j=1, \cdots, l)$
in Definition \ref{D2.1}. 
By \eqref{peak1} and the monotonicity of the function $w$,
we can deduce the existence and uniqueness of local mass centres.

\begin{lemma}\label{L2.1}
Let $u$ be a peak function. 
Then for any point $p_i\in{\bf p}$, there is only one interior local mass centre  $\hat p_i$ in $ B_{{\eps}/4}(p_i)$,
provided that $\bar\delta$ is sufficiently small.
Similarly for any point $q_i\in{\bf q}$, there is only one boundary local mass  centre  $\hat q_i$  on $\pom$ in $ B_{{\eps}/4}(q_i)\cap\Om$.
\end{lemma}

\begin{proof}
We prove that there exists a unique interior local mass centre for $u$ near $p_1$.
For simplicity,  assume that $p_1=0$.

Set $z_t= te_1$, where $e_1$ is the unit vector on the $x_1$-axis.
Denote 
$${\begin{split}
 R_t (u) &=\int_{B_{\eps}(z_t)\cap\{ x_1>t\} } u(x)dx, \\
 L_t (u)&=\int_{B_{\eps}(z_t)\cap\{ x_1<t\} } u(x)dx,
 \end{split}} $$
and $E_t(u)=R_t (u)-L_t(u)$.
Since $w$ is monotone in $|x|$,
by \eqref{peak1}, we conclude that $E_t>0$ for $t\in [-\eps, -c_0 \eps]$ and $E_t<0$ for $t=[c_0\eps, \eps]$,
provided that $\bar\delta$ is sufficiently small,
where $c_0>0$ is a constant with $c_0\to 0$ as $\bar\delta\to 0$.

 Moreover, since $w$ is concave near $x=0$,
it is easy to see that $\frac{d}{dt}  E_t (u) <0$ for $t\in (-c_0\eps, c_0\eps)$. 
Hence, there is a unique point $t_1\in (-c_0\eps, c_0 \eps)$
such that $E_{t_1}(u) =0$.

Repeating the above argument for $z_t=te_i$, $i=2, \cdots, n$, and defining $E_t$ accordingly,
we obtain $t_i$ such that $E_{t_i}=0$. 
Hence the point $(t_1, \cdots, t_n)$ is the unique interior local mass centre of $u$ near $p_1$.
\end{proof}

By Lemma \ref{L2.1}, 
the local mass  centres are uniquely determined for functions satisfying \eqref{peak1},
if $\bar\delta$ is sufficiently small.
Denote the local mass  centres by 
$\hat {\bf p} = (\hat p_1, \cdots, \hat p_k) $ and $\hat{\bf q}=(\hat q_1, \cdots, \hat q_l)$.
Then $\hat {\bf p}$ and $\hat {\bf q}$ depend continuously on $u$, and
they are very close to ${\bf p}$ and ${\bf q}$.
In fact, by a blow-up argument, it follows from \eqref{peak1a}  that 
$|\hat p_1-p_1|=o(\eps)$ as $\bar\delta\to 0$ and $\eps\to 0$ 
(for if not, then $0$ is not the mass centre of $w$).
Thus, in Definition \ref{D2.1}, we may assume that $\hat {\bf p} = {\bf p} $ and $\hat {\bf q} = {\bf q} $,
by replacing ${\bar\delta}$ in \eqref{peak1} by a different small constant if necessary.
Therefore, we can assume that ${\bf p}$ and ${\bf q}$ are local mass  centres of $u$,
such that (i)-(iii) of Definition \ref{D2.1} hold simultaneously.

%% \begin{lemma}\label{L2.3}
%% For any bounded domain $\Om\subset\R^n$ and any positive function $u\in L^1(\Om)$,
%% there is a unique point $z=z_u\in\R^n$ such that
%% \beq\label{cen2}
%% \int_{\Om} (x-z)\, u(x) dx=0.
%% \eeq
%% \end{lemma}

%% \begin{proof}
%% If there are two points $z_1, z_2$ such that \eqref{cen2} holds. 
%% By a rotation of coordinates we may assume that $z_1=z_2+ ae_1$. 
%% Then it implies that $\int_{\Om} au(x)=0.$  
%% \end{proof}

\subsection{The coefficients $a_i, b_j$}\label{S2.3}

%{\color{blue}
We cannot determine the constants $a_1, \cdots, a_k$ and $b_1, \cdots, b_l$ only by \eqref{peak1}.
In order that they are uniquely determined, we let them be the minimizer for 
$$\inf \int_\Om \Big[ u(x)-{\Small\text{$\sum_{i=1}^k$}} a_i w\big( {\Small\text{$\frac{x-p_i}{\eps}$ }}\big)
        - {\Small\text{$\sum_{j=1}^l $}} b_j w\big( {\Small\text{$\frac{x-q_j}{\eps}$ }}\big)\Big]^2 dx ,
$$
where the infimum is taken for all $a_1, \cdots, a_k, b_1, \cdots, b_l\in (1-\hat \delta, 1+\hat\delta)$. 
The integral is a quadratic polynomial of $a_1, \cdots, a_k$ and $b_1, \cdots, b_l$.
Hence there is a unique minimizer for the above infimum. 
For any given small $\hat \delta>0$, the infimum will be attained in the
interval $(1-\hat \delta, 1+\hat \delta)$  provided that $\bar \delta>0$ is small enough.
%}

%% denote
%% $$\bar w=:\int_{\R^n} w(x) dx .$$
%% Then, define
%% $${\begin{split}
 %% a_i & =\frac{1}{\eps^n \bar w} \int_{B_{{\delta}}(p_i)} u(x) dx,\ \ \ i=1, \cdots, k,\\
 %% b_j &= \frac{2}{\eps^n \bar w} \int_{B_{{\delta}}(q_i)\cap \Om} u(x) dx,\ \ \ j=1, \cdots, l, 
%%  \end{split}} $$
%% where ${\delta}$ is the constant in Definition \ref{D2.1}.
%% Thus, the coefficients  $a_1, \cdots, a_k$ and $b_1, \cdots, b_l$ are unique for $u$.

\subsection{Local maximum points}\label{S2.5}
Let $p$ be a peak in Definition \ref{D2.1}. 
A point $z\in B_{\delta}(p)$ is called a local maximum point near $p$ if 
$u(z)=\max\{ u(x)\ |\ x\in B_{{\delta}}(p)\cap\Om\}$. 
%By a local maximum point near $p$,
%we mean a point $z\in B_{\delta}(p)$ such that 
%$u(z)=\max\{ u(x)\ |\ x\in B_{{\delta}}(p)\cap\Om\}$. 

We introduced the local mass  centres for a peak function in Definition \ref{D2.1}.
The local mass centres are not maximum points in general. 
A natural question is whether one can take the local maximum points of $u$ as the peaks
in the definition, namely whether one can replace (ii) in Definition \ref{D2.1} 
by 
\begin{itemize}
\item [(ii)$'$] $p_1, \cdots, p_k $ and $q_1, \cdots, q_l$
are local maximum points of $u$.
\end{itemize}

For the Laplace operator ($p=2$), one can replace (ii) in Definition \ref{D2.1} by (ii)$'$.
However, for $p\ne 2$, the regularity theory of the $p$-Laplace equation
is not strong enough to allow us to replace  (ii)  by (ii)$'$.

\begin{lemma}\label{L2.2}
For any given constant $\lambda\in(0, 1)$, 
there is no solution $u$ to \eqref{NP} of the form \eqref{peak1} with finite energy,
such that
\beq\label{dN}
\lambda\eps<\text{dist}(q, \pom)<\lambda^{-1}\eps 
\eeq
as $\eps\to 0$, where $q$ is a local maximum point of $u$.
\end{lemma}

\begin{proof}
If the conclusion is not true, there exist a sequence $\eps_j\to 0$ 
and functions $u_j$ such that \eqref{peak1} holds 
and its maximum point $q_j$ satisfies \eqref{dN}.
Let $y=(x-q_j)/\eps_j$.
By \eqref{peak1} and the regularity of the $p$-Laplace equation, 
$u_j$ converges locally uniformly to a function $v$ in the coordinates $y$.
From \eqref{dN}, and rotating the coordinates if needed,
the domain $\{y=(x-q_j)/\eps_j\ |\ x\in\Om\}$ converges to a half-space
$\{y_n>-a\}$ for some constant $a\in [\lambda, \lambda^{-1}]$.
Since $u_j$ is a solution to \eqref{NP}, the limit $v$ satisfies the Neumann condition
$v_\nu=0$ on $\{y_n=-a\}$. 
By even extension, $v$ is defined in the entire space $\R^n$.
We derive from \eqref{peak1} that
%By \eqref{peak1}, we have
\beq\label{diff}
|v-w|\le \bar\delta.
\eeq
But $v$ has finite energy. Hence, by the regularity of the $p$-Laplace equation, $v\to 0$ as $x\to \infty$.
Therefore, $v(x)=w(x-p)$ for some point $p$.
By our construction,  $v$ has at least two local maximum points, but $w$ has a unique maximum point, 
which leads to a contradiction.
\end{proof}
 
Note that the above argument also implies that the Fréchet derivatives of the functional 
\beq\label{Ft0}
\mathcal K (u) = \frac{ \eps^{-n}}p \int_\Om \Big[ \eps^p |Du|^p  + |u|^p\Big]dx
                 -\frac{\eps^{-n} }{q}\int_\Om |u|^{q} dx ,\ \ \ u\in W^{1, p}(\Om)
\eeq
do not vanish at a function $u$, 
if $u$ satisfies \eqref{peak1}, and one of its maximum points satisfying \eqref{dN}.

\begin{corollary}  
Let $u=u_\eps$ be a peak solution to \eqref{NP} with finite energy.
Let $p=p_\eps\in\Om$ be a local maximum point of $u$.
Then either $d_p/\eps\to 0$, or  $d_p/\eps\to \infty$,  as  $\eps\to 0.$
\end{corollary}

\begin{corollary}  
Let $u=u_\eps$ %be a peak solution to \eqref{NP} with finite energy. Let 
and $p=p_\eps\in\Om$ be as above. %a local maximum point of $u$.
Then 
\beq\label{es}
\sup_{B_{\delta}(p) } \Big|u(x)-w\Big({\Small\text{$\frac{x-p}{\eps}$}}\Big)\Big|\to 0. 
\eeq
\end{corollary}

Letting $y=\frac{x-p}{\eps}$, then $u_\eps(y)$ converges to a limit $v(y)$.
As in the proof of Lemma \ref{L2.2}, 
$v$ is either a solution to \eqref{NP} in the entire space $\R^n$ or in half space.
In the former case, $v=w$. In the latter case, $v=w$ after an even extension.
Hence, we obtain \eqref{es}.

Let $u$ be a multi-peak solution to \eqref{NP} with finite energy.
It is unclear to the authors 
whether a local maximum point of $u$ near $\pom$ must be on $\pom$.
It may not be true for equation with variable coefficients, even for the simplest model 
\beq\label{NPv}
{\begin{split}
-\eps^2 \Delta u & = b(x) u^{q-1} -  u \  \ \text{in}\ \Om, \\
u_\nu   &=0  \ \ \ \ \text{on}\ \pom , \\ 
u  &>0 \ \ \ \  \text{in}\  \Om,
\end{split}}
\eeq
where $b\in C^0(\bom)$ with  $b>0$.
If $b$ is not Lipschitz continuous, such as $b(x)=1+\sqrt {d_x}$,
the maximum point of a peak solution, 
including the least energy solution,  may not lie in the boundary.
To illustrate this, it suffices to compute $\frac{d}{ds}\I\big(w(\frac{x-z_s}{\eps})\big)$, where 
$z\in\pom$ and $z_s\in \Om$ such that $d_{z_s}=|z_s-z|=s$,
and note that the solution is very close to $w$, 
where $\I$ is the functional given in \eqref{Ft}.

\vskip25pt

\section{A new minimax principle}\label{S3}

\vskip10pt

\subsection{A new minimax principle}\label{S3.1}
Denote
\beq\label{Ft}
\I(u)=\frac{\eps^{-n} \int_\Om (\eps^p\, |Du|^p+|u|^p)}{\big[\eps^{-n} \int_\Om |u|^{q}\big]^{p/q}},\ \ \ u\in W^{1,p}(\Om).
\eeq
For any integers $k\ge 0, l\ge 0$ ($k+l\ge 1$), 
let the notations {\bf p}, {\bf q},  {\bf a}, {\bf b} and  $\Phi_{k,l} ({\bf p, q},  {\bf a, b})$ be as in \S\ref{S2.1}.
Denote
\beq\label{Spq}
S({\bf p, q},  {\bf a, b})=\inf_{u\in \Phi_{k,l} ({\bf p, q},   {\bf a, b})} \,\I(u) .
\eeq
%% denote $S({\bf p,a})= S({\bf p,a},  {\bf q,b})$ if $l=0$.
Let
\beq\label{S*}
S^*=\sup_{({\bf p,q})\in \Om^k\times \pom^l, \, ({\bf a,b})\in M} \ S({\bf p, q},  {\bf a,b}), 
\eeq
where
${\bf p, q}$ satisfy the $\delta$-apart condition, 
$$
M  =  \big\{(a_1, \cdots, a_k, b_1,\cdots, b_{l})\in\R^{k+l} \ |\   |a_i-1|<\hat\delta, |b_j- 1|<\hat\delta\big\}.
$$ 
  
We have the following minimax principle,
which implies Theorem \ref{T1.1}.

\begin{theorem}\label{T3.1}
For sufficiently small $\eps>0$, the minimax $S^*$ given in \eqref{S*} is a critical value 
of the functional $\I$, and the corresponding critical point is a peak solution to \eqref{NP}
with $k$ interior peaks and $l$ boundary peaks.
\end{theorem}

%% {\it Geometric characterisation of the solution}. 
When $p=2$, multi-peak solutions were obtained in \cite{CK99,  DY99, DY99b, Gui96, GW00, GWW, Wa95}, 
among many others.
Let $u$ be the solution obtained in Theorem \ref{T3.1}
with peaks {\bf p, q} and coefficients {\bf a, b}.
Similarly to \cite{DY99}, we have the following properties:
$${\begin{split}
 & d_{p_i}/\eps\to\infty,\ \ 1\le i\le k,  \\
 &  |p_i-p_j|/\eps\to \infty,\ \ i\ne j, \ \ \text{(if $k>1$)},\\
 &  |q_i-q_j|/\eps\to \infty, \ \ i\ne j,\ \ \text{(if $l>1$)}, \\
 & a_i\to 1, b_j\to 1, \ \ 1\le i\le k,\ \ 1\le j\le l,
 \end{split}}$$
as $\eps\to 0$.  
These properties are easily seen from our argument 
%{\color{blue} 
(by Lemmas \ref{L3.1}, \ref{L3.4}, \ref{L3.5}), %}, 
because the solution takes the supremum in \eqref{S*}.

\subsection{Remarks}\label{S3.2}

Denote by $\Phi(q)$ the set of peak functions with a single peak at $q$.
When $k=1$ and $l=0$,  the minimax \eqref{Spq}-\eqref{S*} becomes
\beq\label{S*1}
S^*=\sup_{q\in\Om_{{\delta}}} \ \inf_{u\in \Phi(q) } \,\I(u) .
\eeq
When $k=0, l=1$,  \eqref{Spq}-\eqref{S*} becomes
\beq\label{S*2}
S^*=\sup_{q\in\pom} \ \inf_{u\in \Phi(q) } \,\I(u),
\eeq
In both cases above, 
the functional $\I(u)$ is independent of the coefficients {\bf a} and ${\bf b}$,
since $\I(u)$ is homogeneous.
In \eqref{S*1}, the supremum is attained over $q\in\Om_\delta$, 
because we assume that $d_q>\delta$ by the $\delta$-apart condition in Definition \ref{D2.1}.

From \eqref{S*1} and \eqref{S*2}, our solution is not the least energy solution.
For a given point $z\in\bom$, the minimizer for the infimum \eqref{S*1} 
corresponds to the Mountain Pass Solution by restricting the functional $\I$ 
to the Banach space 
$$W^{1,p}_q(\Om)=\{u\in W^{1,p}(\Om)\ |\ \text{the mass centre of $u$ is located at the point $q$}\}. $$
The supremum in \eqref{S*1} indicates that the critical point has the highest energy
among all such Mountain Pass Solutions.
When $k=0, l=1$, according to \cite{NT91},
our solution should have its peak {\it near} the boundary point 
with the smallest mean curvature if the domain is convex,
or with the largest negative mean curvature otherwise, at least in the case $p=2$.
We refer the reader to \cite{LZ05a, LZ07}, 
where the least energy solution in the case $p\ne 2$ was obtained and it was shown that 
the maximum point of the solution is very close to the boundary point 
with maximal mean curvature. 
However, it remains unclear whether the maximum point is located on the boundary.

\vskip5pt
\noindent {\it Remark on the coefficients {\bf a}, {\bf b}}.\\
({\it{i\,}}) When $k=2$ and $l=0$, one has ${\bf p}=(p_1, p_2)$.
For any $u\in \Phi_{k, l}({\bf p,q,a,b})$, by definition, there exist functions $u_1\in \Phi(p_1)$ and $u_2\in \Phi(p_2)$
such that $u=u_1+u_2$ (more precisely, $u=u_1+u_2+o(1)$, where the term $o(1)$ can be absorbed into either  $u_1$ or  $u_2$).
Hence, we can formally write 
\beq\label{dec}
\Phi_{k, l}({\bf p,q,a,b})=\Phi(p_1)\oplus\Phi(p_2).
\eeq
This decomposition becomes more obvious 
when $\Om$ is the union of two disjoint sub-domains $\Om_1$ and $\Om_2$, 
with $p_1\in\Om_1$, $p_2\in\Om_2$.
In this case, the solution we seek consists of two parts,
each being a single-peak solution localized in $\Om_1$ and $\Om_2$, respectively.
%which are single peak solutions respectively in $\Om_1$ and $\Om_2$.
To balance these two single peak solutions, 
we introduce the coefficients $a_1, a_2$. 
The concavity property discussed in \S\ref{S3.5} suggests that one should look for maximizers for $a_1, a_2$ near 1. 
\\[5pt]
({\it{ii\,}}) For general $k, l\ge 1$,  a similar decomposition to \eqref{dec} holds.
The role of the coefficients $a_1, \cdots, a_k$ and $b_1, \cdots, b_k$ is
to balance the size of the solution near each peak.

 \vskip5pt
\noindent {\it Remark on the functional $\I$}.\\
Instead of the functional $\I$, we may consider the functional $\mathcal K$ in \eqref{Ft0}.
In analogy with the Mountain Pass Lemma, 
we introduce   
$$Q({\bf p, q},  {\bf a, b})=\inf_{u\in \Phi_{k,l} ({\bf p, q},  {\bf a, b})}  \sup_{s>0} \mathcal K(su), $$
and corresponding to \eqref{S*}, we define 
$$Q^*=\sup_{({\bf p,q})\in \Om^k\times \pom^l, \, ({\bf a,b})\in M} \ Q({\bf p, q},  {\bf a,b}). $$
%{\color{blue} 
The functional $\mathcal K$ allows us to deal with the Neumann problem \eqref{NP}
with more general right hand side functions $f(u)$,
namely replacing the function $u^{q-1}-u^{p-1}$ by a more general function $f(u)$, 
assuming that $h(s)=:\mathcal K(sw)$ has a {\it unique} maximum point $\hat s=\hat s_w>0$ 
and that $h$ is strictly concave at $\hat s$, for the ground state solution $w$.
There is no loss of generality in assuming that $\hat s=1$.
One also assumes that 
equation \eqref{ST00} with right hand side $f(u)$ has a unique ground state solution $w$.
To carry out our argument to the functional $\mathcal K$, 
one shows that for a given initial condition $u_0$ in $\Phi_{k,l} ({\bf p, q},  {\bf a, b})$ such that $\hat s_{u_0}=1$,  
there is a deformation $u_0\to u_t\ (t>0)$
such that  $\hat s_{u_t}=1$ and the functional $\mathcal K(u_t)$ is decreasing.
We will treat the functional $\mathcal K$ in a separate paper.
%}

%% For the functional $\mathcal K$, one should design a gradient flow according to the functional.
%% For the functional $\mathcal K$, the gradient flow needs modification, such that for any $t>0$, 
%% $u_s(\cdot, t)\approx s u_{s_t}(\cdot, t)$.
%% More precisely, for any given $t>0$, we have the curve $u_s(\cdot, t)$ for $s\in (0, 1)$.
%% Assume that $\sup_s \mathcal K(u_s(\cdot, t)$ is attained at $s_t$ (for a unique $s_t$). 
%% Then there is a 1-1 corresponding between $u_s(\cdot, t)$ and $s u_{s_t}(\cdot, t)$.

\subsection{Strategy of the proof}\label{S3.3}
The set of peak functions is only a small subset of the Sobolev space $W^{1, p}(\Om)$,
and it may not be an open set in $W^{1, p}(\Om)$.
Therefore, the usual minimax principle cannot be applied directly.
To overcome this difficulty, we introduce a parabolic $p$-Laplace equation as a gradient flow for the functional $\I$ defined in \eqref{Ft}, 
and establish the existence of critical points of $\I$  by 
combining the gradient flow with topological degree theory.

In the case $k=1, l=0$, 
we will consider a family of the gradient flows given by \eqref{PE}-\eqref{PEb} simultaneously,
with initial condition $\phi_p=:w(\frac{x-p}{\eps})$ for all points $p\in\bom$.
For any time $t>0$, the solution $u_p(\cdot, t)$ is a peak function 
if $\I(u_p(\cdot, t))>S^*-\sigma$,
where $\sigma$ can be any small positive constant (in particular, we choose $\sigma=e^{-1/\eps^2}$).
If there exists a time $t^*$ such that $\I(u_p(\cdot, t))=S^*-\sigma$, 
we freeze the flow by setting $u_p(\cdot, t)=u_p(\cdot, t^*)$ for all $t>t^*$.
 %($\sigma$ can be any small positive constant, where we choose $\sigma=e^{-1/\eps^2}$).
%Whenever $\I(u_p(\cdot, t))=S^*-\sigma$ at some time $t^*$, 
%we freeze the flow, namely we set $u_p(\cdot, t)=u_p(\cdot, t^*)$ for all $t>t^*$.
In this way, the solution $u_p(\cdot, t)$ is well-defined for all $t>0$.

Denote $z_{p, t}$ the local mass centre of $u_p(\cdot, t)$,
which is unique according to Lemma \ref{L2.1}.
Define the mapping $T_t:\ p\to z_{p, t}$. 
Then, $T_t$ is continuous in both $p$ and $t$.
We show that $\I(\phi_p)<S^*-\sigma$ if $d_p\le \delta$,
where $\delta=N\eps$ for some fixed large constant $N>1$.
It follows that $T_t$ is the identity mapping at $t=0$ and on $\pom_{\delta}$ for all $t\ge 0$.
Moreover, we prove that whenever $z_{p, t}$ is close to $\pom_{\delta}$, it holds that 
$\I(u_p(\cdot, t))\le S^*-\sigma$, and thus the flow is frozen.

Therefore, for all $t>0$, by the Brouwer degree theory, 
$T_t(\bom_{\delta})$ is a cover of $\bom_{\delta}$.
By the definition of $S^*$, for any time $t>0$, there is a point $p^*=p^*_t\in\bom_{\delta}$ 
such that $\I(u_{p^*}(\cdot, t))>S^*$.  Since \eqref{PE}-\eqref{PEb}  is a descent gradient flow,
one sees that $\I(u_{p^*}(\cdot, t'))>S^*$ for all time $t'<t$.
Assume that the local mass centre $z_{p^*, t}$ converges to a point $z^*\in\Om$ as $t\to\infty$.
Then $\I(u_{z^*}(\cdot, t))>S^*$ for all time $t>0$,
and $u_{z^*}(\cdot, t)$ converges to a solution of \eqref{NP} with one interior peak $z^*$.

The proof for general $k$ and $l$ follows a similar idea. However, we consider a larger family 
of the gradient flows \eqref{PE}-\eqref{PEb} simultaneously, 
with initial condition given in \eqref{weps}.

\subsection{Ground state solution to \eqref{ST00}}\label{S3.4}

It was shown in \cite{ST00} that, for $1<p<n$, equation \eqref{ST00} admits a unique ground state solution $w$, 
i.e., a positive solution that decays to zero at infinity and is rotationally symmetric.
%in the case $1<p<n$.  Moreover, 
Furthermore, 
a precise exponential decay estimate for $w$ was established in \cite{LZ05b}, namely,
\beq\label{PD}
\lim_{r\to\infty} w(r) r^{\frac{n-1}{p(p-1)}} e^{(\frac{1}{p-1})^{\frac{1}{p}}r} =C
\eeq
for a positive constant $C$.

Here we present a simple proof of the exponential decay 
(but not as strong as \eqref{PD}),
which may apply to more general right hand side of \eqref{ST00}.
Let $w$ be a radial solution, written as $w=w(r)$.
Then,  equation \eqref{ST00} can be rewritten in the form
\beq\label{ST01}
|w'|^{p-2} \big[ w''+\frac{n-1}{p-1} \frac {w'}{r}\big] = \frac{1}{p-1}\big(w^{p-1} -  w^{q-1} \big)\ge \frac  {w^{p-1}}{p},
\eeq
when $w>0$ is small.
Multiplying by $pw'$, we get
$$
\big( |w'|^p-\frac 1p w^p\big)' + \frac{p(n-1)}{(p-1)r} |w'|^p\le 0.
$$
Hence, $\psi=:|w'|^p-\frac 1p w^p$ is decreasing.
Since $w\to 0$  as $r \to \infty$, we must have $\psi\to 0$.
Otherwise, if $\psi\to c^2>0$, then $w'\to -c$,
which contradicts the fact that $w\to 0$. 
Thus, $\psi\ge 0$, and we conclude that 
\beq\label{dec1}
w'+ a w\le 0, 
\eeq
where $a=:\frac{1}{p^{1/p}}$. 
It follows that
$e^{ar} w(r)$ is decreasing. Hence 
\beq\label{dec2}
w(r)\le C_1e^{-ar}.
\eeq
The proof of \eqref{dec2} follows the argument in \cite{GNN}.

From equation \eqref{NP}, one has
$\Delta_p w <  w^{p-1}$.
By constructing a suitable sub-solution, we have the boundary gradient estimate  in the domain  $B_{r+1}-B_r$,
$$ |w'(r)|\le C_2w(r) .$$
By  \eqref{dec1}, $ \frac{w}{|w'|}\le C$. 
Hence, we get
$$-\frac{w'}{w}=h(x)$$
for a positive and bounded function $h$. Taking integration, we obtain  
\beq\label{dec3}
w(r)\ge e^{-br} . 
\eeq
Estimates \eqref{dec2} and \eqref{dec3} show the exponential decay of the solution $w$.  
  
\vskip10pt

Let
\beq\label{Ft1}
\J(u) = \frac{  \int_{\R^n} (  |Du|^p+|u|^p)}{\big[  \int_{\R^n}  |u|^{q}\big]^{p/q}} ,\ \ \ u\in W^{1,p}(\R^n) ,
\eeq
and  let
\beq\label{S0}
S_0=\inf \,  \{ \J(u)\ |\ {u\in W^{1,p}(\R^n)}, u\ne 0\}
\eeq
be the best constant for the Sobolev embedding $W^{1, p}(\R^n) \hookrightarrow L^q(\R^n) $ with the norm 
$$\|u\|_{W^{1,p}(\R^n)}= : \Big| \int_{\R^n} (  |Du|^p+|u|^p)\Big|^{1/p} .$$
The infimum $S_0$ is attained by the ground state solution $w$ to equation \eqref{ST00}. 
It follows from \eqref{ST00} that
$$\int_{\R^n}  (  |Dw|^p+w^p) =  \int_{\R^n} w^q, $$
which implies that
\beq\label{S00}
S_0=\J(w)= \Big[ \int_{\R^n}  (  |Dw|^p+w^p)\Big]^{1-\frac pq}. 
\eeq

 \subsection{A concavity property}\label{S3.5}
Denote
\beq\label{tef}
\bar{w}(x) 
={\Small\text{$ \sum_{i=1}^k$}}  \, a_iw\big( {\Small\text{$\frac{x-p_i}{\eps}$}} \big)
             + {\Small\text{$\sum_{j=1}^l$}} \, b_j w\big({\Small\text{$\frac{x-q_j}{\eps}$}} \big).
\eeq
Assume that ${\bf p, q}$ satisfy the $\delta$-apart condition. Then, 
$$
\I(\bar{w}) = \frac{\eps^{-n} \big(\sum_{i=1}^k \int_{B_\delta(p_i)}+\sum_{j=1}^l \int_{B_\delta(q_j)\cap\Om} +\int_\omega\big)  [\eps^p |D\bar{w}|^p+|\bar{w}|^p] }
                  {\big[\eps^{-n} \big(\sum_{i=1}^k \int_{B_\delta(p_i)}+\sum_{j=1}^l \int_{B_\delta(q_j)\cap\Om} +\int_\omega\big)  |\bar{w}|^q\big]^{p/q} }  ,
$$
where $\omega=\Om-\{ \cup_{i=1}^k B_\delta(p_i) \}-\{ \cup_{j=1}^{l} B_\delta(q_j) \}$. 
For the integrals over $B_\delta(p_i)$ and $B_\delta(q_j)\cap\Om$, we apply the change of variables $y=\frac{x-p_i}{\eps}$ and $y=\frac{x-q_j}{\eps}$, respectively.
By the exponential decay of the ground state solution $w$, we have
\beq \label{Ipq}
{\begin{split}
\I(\bar{w}) &=\frac{\sum_{i=1}^k a_i^p \int_{B_R(0)} [ |Dw|^p+w^p]+\sum_{j=1}^l b_j^p \int_{B^+_R(0)}  [ |Dw|^p+w^p]+o(1)}
                  {\big[ \sum_{i=1}^k a_i^q \int_{B_R(0)} w^q+\sum_{j=1}^l b_j^q \int_{B^+_R(0)} w^q+o(1)\big]^{p/q} } \\
       & = H({\bf a, b}) S_0 +o(1)
\end{split}} \eeq 
for small  enough$\eps$, 
where $R=\delta/\eps$, $B_R^+(0)=B_R(0)\cap\{x_n>0\}$, $o(1)\to 0$ as $\eps\to 0$ and $R\to\infty$, 
$S_0$ is given in  \eqref{S00}, and 
\beq\label{Hab}
 H({\bf a, b})=:
 \frac{\sum_{i=1}^k a_i^p +\frac 12\sum_{j=1}^l b_j^p } {\big[ \sum_{i=1}^ka_i^q + \frac 12\sum_{j=1}^l b_j^q\big]^{p/q}} .
 \eeq
For brevity, write the functional $\I$ in \eqref{Ipq} as 
 \beq\label{Iuab}
   \I(\bar{w}) 
   =  \frac{a_1^p\int_{B_R(0)} \big[|Dw|^p+w^p\big] + \alpha}{\big[ a_1^q \int_{B_R(0)}w^q+\beta \big]^{p/q}},
\eeq
where $\alpha, \beta$ denote the rest parts in the formula.
We  compute, when $R$ is sufficiently large and $\eps$ is suffiiciently small, 
$${\begin{split}
\p_{a_1}  \I(\bar{w}) 
    & \approx  \frac{pa_1^{p-1} \int_{B_R(0)} [|Dw|^p+w^p] }{[a_1^q \int_{B_R(0)}w^q+\beta ]^{p/q}}
       - \frac{pa_1^{q-1} \big[a_1^p\int_{B_R(0)} [|Dw|^p+w^p] + \alpha\big] \int w^q}{[a_1^q \int_{B_R(0)}w^q+\beta]^{1+p/q}}.
\end{split}}$$
In the above computation, we use the special form of $\bar{w}$ in \eqref{tef}.
It means the derivatives of $\alpha, \beta$ in $a_1$ is very small and negligible.

Recall that 
$\int_{\R^n} [|Dw|^p+w^p] = \int_{\R^n} w^q$. 
Hence
$${\begin{split}
\p_{a_1}  \I(\bar{w}) 
   &\approx  \frac{pa_1^{p-1} \int w^q }{[ a_1^q \int_{B_R(0)}w^q+\beta]^{p/q}}
       - \frac{pa_1^{q-1} \big[a_1^p\int_{B_R(0)} [|Dw|^p+w^p] + \alpha \big] 
                                            \int w^q}{[a_1^q \int_{B_R(0)}w^q+\beta ]^{1+p/q}}\\
   &=  \frac{pa_1^{p-1} \int w^q}{[a_1^q \int_{B_R(0)}w^q+\beta ]^{p/q}}
      \Big[ 1-  \frac{a_1^{q-p}\big[a_1^p \int_{B_R(0)} [|Dw|^p+w^p] + \alpha\big]}{ a_1^q \int_{B_R(0)}w^q+\beta }\Big].
      \end{split}}$$
From the above formula, it is easy to see that when $a_1=\cdots=a_k=1$, $b_1=\cdots = b_l=1$,
$\nabla_{a_1} \I(\bar{w})  \approx 0$
if $\eps$ is sufficiently small and $R=\delta/\eps$ is sufficiently large.
Similarly, one has $\nabla_b \I(\bar{w})  \approx 0$. Therefore, we obtain
\beq\label{nab}
\nabla_{a, b} \I(\bar{w})  \approx 0 \ \ \ \text{if}\ \ a_i= b_j=1 \ \mbox{for any} \ 1\leq i\leq k, \ 1\leq j\leq  l. %\ \forall\ i, j.
\eeq
In fact,  \eqref{nab} follows from the fact that $w$ is a ground state to \eqref{ST00}.

The following lemma states that if first derivative of $H$ vanishes at some point {\bf a},{\bf b}, 
then this point is a local maximum point of $H$.

\begin{lemma}\label{L3.1} 
Assume that $k+l>1$. 
The function $H$ has the following properties
\begin{itemize}
\item [(a)] If $\p_{a_i} H({\bf a, b})=0$, then $\p^2_{a_i} H({\bf a, b})<0$, 
\item [(b)] If $\p_{b_j} H({\bf a, b})=0$, then $\p^2_{b_j} H({\bf a, b})<0$, 
\item [(c)] If $\nabla H({\bf a, b})=0$, then $\nabla^2_\xi H({\bf a, b})<0$
\end{itemize}
for any vector $\xi\nparallel ({\bf a, b})$,
where $1\le i\le k$, $1\le j\le l$.
\end{lemma}

\begin{proof}
For any given $1\le i\le k$, denote $t=a_i$,
$${\begin{split}
\alpha &=\big( {\Small\text{$\sum_{j=1}^k$}} a_j^p + {\Small\text{$\frac 12\sum_{j=1}^l $}}b_j^p \big)-a_i^p\ge 0, \\
\ \ \beta &= \big( {\Small\text{$\sum_{j=1}^k$}} a_j^q+ {\Small\text{$\frac12 \sum_{j=1}^l $}} b_j^q\big)-a_i^q\ge 0 .
\end{split}}$$
The notations $\alpha, \beta$ are not the same as in \eqref{Iuab}. Set
$$f(t)=: H({\bf a, b})=\frac{\alpha + t^p } {[ \beta+t^q ]^{p/q}}. $$
We calculate
$$f'(t)=\frac {pt^{p-1} } {[ \beta+t^q]^{p/q}} -\frac{p(\alpha + t^p)t^{q-1} } {[ \beta+t^q]^{1+p/q}} .
$$
If $f'(t)=0$, we have
$$ \beta+t^q=(\alpha+t^p) t^{q-p}. 
$$
By the homogeneity of the function $H$, we may assume that $f'(t)=0$ at $t=1$. Hence $\alpha=\beta$, and
\beq\label{f1}
{\begin{split}
f'(t) &=\frac {pt^{p-1} } {[ \alpha+t^q]^{p/q}} -\frac{p(\alpha + t^p)t^{q-1} } {[ \alpha+t^q]^{1+p/q}} \\
     &= \frac {\alpha p(t^{p-1} -t^{q-1})} {[ \alpha+t^q]^{1+p/q}} .
\end{split}}
\eeq
We calculate the second derivative of $f$, 
$${\begin{split}
f''(t) &=  \frac {\alpha p((p-1)t^{p-2} -(q-1)t^{q-2})} {[ \alpha+t^q]^{1+p/q}}
    -  \frac {\alpha p(p+q)(t^{p-1} -t^{q-1})t^{q-1}} {[ \alpha+t^q]^{2+p/q}}.
    \end{split}}
 $$
At $t=1$,  if $\alpha>0$ (which requires $k+l>1$), then
$$
f''(t)  =  \frac {\alpha p((p-1)t^{p-2} -(q-1)t^{q-2})} {[ \alpha+t^q]^{1+p/q}}{\Big |}_ {t=1} 
    = \frac {\alpha p(p-q)} {[ \alpha+1]^{1+p/q}} <0 . $$
Hence, we obtain part (a)  of Lemma \ref{L3.1}. Part (b) follows by a similar computation. From %the same computation.
 \eqref{f1}, we observe that the Hessian matrix $\nabla^2 f$ is diagonal  at the point where $\nabla f=0$.
Therefore, part (c) also holds.
%Hence we also obtain (c).    
\end{proof}

We remark that Lemma \ref{L3.1} relies on the assumption $\nabla H=0$;
 the coefficient $\frac 12$ in \eqref{Hab} does not play any essential role in the proof. 
Note that $H ({\bf a, b})$ is a constant in the ray $t({\bf a, b})$ $(t>0)$.
Hence in Lemma \ref{L3.1} (c), we assume that $\xi\nparallel ({\bf a, b})$,
namely $\xi$ is not parallel to the vector $({\bf a, b})$.
We also note that \eqref{nab} holds because $w$ is the ground state solution to \eqref{ST00}.

Due to the smoothness of the function $H({\bf a, b})$, Lemma \ref{L3.1}  implies
that there exists a maximizer when taking supremum for the parameters {\bf a, b} in \eqref{S*}.

\subsection{Estimate for $S^*$}\label{S3.6}
To prove Theorem \ref{T3.1}, we need the following estimate for $S^*$.

\begin{theorem}\label{T3.2}
We have
\beq\label{es0}
\lim_{\eps\to 0 }  S^* =  S_{0, k, l}, 
\eeq
where
\beq\label{S0kl}
S_{0, k, l} =(k +{\Small\text{$\frac l2$}} )^{1-p/q}  S_0.
\eeq
\end{theorem}

The proof of \eqref{es0} is divided into
\beq\label{es0a}
\underset{\eps\to 0} {\overline\lim} \, S^* \le  S_{0, k, l} 
\eeq
and 
\beq\label{es0b}
\underset{\eps\to 0} {\underline\lim} \,  S^* \ge  S_{0, k, l} .
\eeq

\noindent{\bf Remark}.
To prove \eqref{es0a}, 
by the definition of $S^*$ in \eqref{S*}, 
we have to consider all possible points {\bf p}, {\bf q} and coefficients {\bf a}, {\bf  b}.
However, thanks to \eqref{Spq}, it suffices to select a particular peak function $u\in \Phi_{k,l} ({\bf p}, {\bf  q},  {\bf a}, {\bf  b})$.

%but by \eqref{Spq}, we can choose a special peak function $u\in \Phi_{k,l} ({\bf p}, {\bf  q},  {\bf a}, {\bf  b})$.

To prove \eqref{es0b}, on the other hand, 
we can choose special points {\bf p, q} and coefficients {\bf a}, {\bf  b},
but we need to consider the infimum in \eqref{Spq} for all $u\in \Phi_{k,l} ({\bf p}, {\bf  q},  {\bf a}, {\bf  b})$.
This properties will play a crucial role in the following argument.

We begin by introducing a useful lemma.
%First we introduce a very helpful lemma.

\begin{lemma}\label{L3.2}
Let $D_t$ be a deformation of a domain $D_0$,  such that 
$D_s\subset D_t$ for any $t>s\ge 0$.
Denote
 $$
h(t)=\frac { \alpha+\int_{D_t} [ |Du|^p +|u|^p ] }{[ \beta+ \int_{D_t} |u|^q ]^{p/q}}, \ \ u\in W^{1,p}(D_t),
 $$
 where $\alpha, \beta\ge 0$ are constants. 
 Then $ h(t) $ is monotone increasing in $t$, if  $u$ is sufficiently small on $\p D_t$ for $t>0$. 
 \end{lemma}
 
 \begin{proof}
We have
$$
 {\begin{split} 
 \frac{d}{dt} h(t)
    & = \frac{ \int_{\p D_t} [|Du|^p+|u|^p] g(x) }{b^{p/q} } 
        -\frac {a p\int_{\p D_t} |u|^q g(x) } {q b^{1+p/q}}  \\
   &\ge \frac 1{b^{p/q}} \Big[ \int_{\p D_t} |u|^p g -\frac {ap}{bq}  \int_{\p D_t} |u|^q g\Big] , 
\end{split}}
$$
where $a=\alpha+ \int_{\p D_t} [|Du|^p+|u|^p], b=\beta + \int_{D_t} |u|^q$, and
$g\ge 0$ is the velocity of deformation.

Under the assumption that $u$ is  small on $\p D_t$, we have $|u|^p -\frac {ap}{bq}|u|^q\ge 0$ on $\partial D_t$, since $q>p$.
Therefore, $\frac{d}{dt}  h(t)  \ge 0$, which implies that $ h (t)$ is monotone increasing in $t$.
\end{proof}

In the following, we prove \eqref{es0a} and \eqref{es0b} step by step.

\vskip10pt

\subsubsection{The case $k=1, l=0$.}\label{S3.6.1}
In this case, let $p\in\Om$ with $d_p>\delta$. 
Set $y=(x-p)/\eps$ and $\Om^\eps=(\Om-\{p\})/\eps$. Then
\beq\label{Ft2}
\I(u)=\frac{\int_{\Om^\eps} [ |Du|^p +|u|^p ]} {\big[ \int_{\Om^\eps} |u|^q \big]^{p/q} } ,
\eeq
where $\Om^\eps\supset B_{\delta/\eps}(0)$. 
Let $u$ be a peak function with a single peak at $p$.
By extending $u$ to $\R^n$ and applying Lemma \ref{L3.2}, we obtain
$\I(u)\le\J(u)$, where $\J(u)$ is given in \eqref{Ft1}. 
Since $w$ itself is a peak function, \eqref{es0a} holds in this case.

Next, we prove \eqref{es0b}.  Choose a point $p\in\Om$ such that $B_r(p)\subset \Om$, 
for some constant $r>0$ depending only on $n$ and $\Om$.
By the definition of $S^*$ (i.e., the supremum in \eqref{S*}) and the monotonicity in Lemma \ref{L3.2},
it suffices to prove \eqref{es0b} for $\Om=B_r(p)$. By  translation, we may assume $p=0$.
Thus, we consider  the functional
$$\I_r(u)=\frac{\eps^{-n} \int_{B_r(0)} [\eps^p |Du|^p +|u|^p]}{\big[ \eps^{-n}  \int_{B_r(0)} |u|^q\big]^{p/q} } .$$
Denote by $\Phi(0)$ the set of functions with mass centre at $p=0$.
Then \eqref{es0b} follows from the following lemma.

\begin{lemma}\label{L3.3} 
We have
\beq\label{Es1}
\lim_{\eps\to 0 }\inf \{\I_r(u)\ |\ u\in\Phi(0)\} = S_0.
\eeq

 \end{lemma}

\begin{proof}
Let $y=x/\eps$. The functional $\I_r$ is changed to 
$$\J_R (u)=\frac{\int_{B_R(0)} [ |Du|^p +|u|^p]dy}{\big[   \int_{B_R (0)} |u|^q dy\big]^{p/q} } ,$$
where $R=r/\eps$.  Assume that the infimum in \eqref{Es1} is attained by $w_R$. 
Then $w_R$ is a minimizer of $\J_R$ among all functions satisfying \eqref{peak1a} and  \eqref{peak2a}.
Letting $\eps\to 0$, i.e., $R\to\infty$,  we pass to the limit $w_R\to w_\infty$. 
It follows that $w_\infty$ is a minimizer of the limiting functional $\J$ defined in \eqref{Ft1}.
By uniqueness of the ground state solution to \eqref{ST00}, we conclude that $w_\infty=w$, which means \eqref{Es1}.
%Hence $w_\infty=w$, the unique ground state solution to \eqref{ST00}.
\end{proof}

\subsubsection{  The case $k=0, l=1$.} \label{S3.6.2}
In this case, $S^*$ is the critical value of the functional $\I$ corresponding to the least energy solution \cite{LZ05a, LZ07}.
Moreover, estimate \eqref{es0}  follows directly  from the case $k=1, l=0$ discussed in \S \ref{S3.6.1}.

Indeed, given a point $q\in\pom$, 
by a translation and a rotation of the coordinates, 
we assume that $q=0$ and $\{x_n=0\}$ is the tangent plane of $\Om$ at $q$ 
with inner normal $e_n=(0, \cdots, 0, 1)$.
Then locally, $\pom$ is represented by $x_n=\rho(x')$, where $x'=(x_1, \cdots, x_{n-1})$.
We make the coordinate transform
$y'=x'$  and $y_n=x_n-\rho(x')$ to flatten the boundary locally. 
We then extend $u$ to $\{y_n<0\}$ by even extension.
Under this transformation, estimate \eqref{es0} also follows from the case $k=1, l=0$ in \S \ref{S3.6.1}.

\vskip5pt

\subsubsection{  The case $k>1, l=0$.} \label{S3.6.3}
In this case, $  S_{0, k, l} =k^{1-p/q}  S_0$.

First, we prove \eqref{es0b}. As observed after \eqref{es0b}, it suffices to verify \eqref{es0b}
for a particular set of points $p_1, \cdots, p_k$ and coefficients $a_1, \cdots, a_k$.
We choose points $p_1, \cdots, p_k$ such that $d_{p_i}\ge {r}$ for all $i$, $|p_i-p_j|\ge {r}>0$ for all $i\ne j$, $1\leq i,j\leq k$, 
and select $a_1=\cdots=a_k=1$.
By the monotonicity Lemma \ref{L3.2}, 
it is then sufficient to consider the functional $\I$ on the domain 
$\Om=B_{{r}} (p_1)\cup \cdots \cup B_{{r}} (p_k)$.
On each ball, the functions are completely independent. 
Therefore, \eqref{es0b} follows directly from \eqref{Es1}.
%The function in these balls are completely independent. 
%Hence \eqref{es0b} follows from \eqref{Es1}.

Next, we prove \eqref{es0a}.
Choose arbitrary points $p_1, \cdots , p_k\in\Om$ satisfying the $\delta$-apart condition.
Then the balls $B_{\delta} (p_1), \cdots, B_{\delta} (p_k)$ are mutually disjoint and fully contained within $\Om$.
Denote $G=\Om-\big\{ \cup_{i=2}^{k} B_{\delta} (p_i)\}$. 
For any peak function $u \in \Phi_{k,l} ({\bf p, q},  {\bf a, b})$, the monotonicity from Lemma \ref{L3.2} implies that
$${\begin{split}
\eps^{n(1-p/q)} \I(u) 
     & = \frac{\sum_{i=2}^{k} \int_{B_{\delta}(p_i)} [\eps^p|Du|^p+|u|^p] +\int_G [\eps^p|Du|^p+|u|^p] }
                    {\big[\sum_{i=2}^{k} \int_{B_{\delta}(p_i)} |u|^q +\int_G |u|^q\big]^{p/q} }\\
        & \le \frac{\sum_{i=2}^{k} \int_{B_{\delta}(p_i)} [\eps^p|Du|^p+|u|^p] +\int_\Om [\eps^p|Du_1|^p+u_1^p] }
                    {\big[\sum_{i=2}^{k} \int_{B_{\delta}(p_i)} |u|^q +\int_\Om u_1^q\big]^{p/q} }.
                \end{split}} $$    
In the inequality above, we extend $u$ from $G$ to $\Om$, obtaining a peak function $u_1$ with a single peak $p_1$.
Similarly, extending $u$ from $B_\delta(p_2)$ to $\Om$ yields a peak function $u_2$ with a single peak $p_2$.
Applying the monotonicity  Lemma \ref{L3.2}, we deduce
$${\begin{split}
\eps^{n(1-p/q)} \I(u) 
            & \le \frac{\sum_{i=3}^{k} \int_{B_{\delta}(p_i)} [\eps^p |Du|^p+|u|^p] 
                    + \sum_{i=1}^2\int_\Om [\eps^p|Du_i|^p+u_i^p] }
                    {\big[ \sum_{i=3}^{k} \int_{B_{\delta}(p_i)} |u|^q +\sum_{i=1}^2 \int_\Om u_i^q \big]^{p/q}}.
\end{split}} $$
Repeating this argument, we construct peak functions $u_1, \cdots, u_k$  
such that $u_i$ has a single peak $p_i$.
Then, by Lemma \ref{L3.2}, we obtain the inequality
$${\begin{split}
\eps^{n(1-p/q)} \I(u) 
          \le \frac{\sum_{i=1}^k\int_\Om [\eps^p |Du_i|^p+u_i^p] } {\big[ \sum_{i=1}^k\int_\Om u_i^q\big]^{p/q}}.
\end{split}} $$
As remarked after Theorem \ref{T3.2}, 
we can choose a special function $u$ for the proof of \eqref{es0a}.
Here, we choose $u_i=a_i w\big(\frac{x-p_i}{\eps}\big)$, 
where $w$ is the ground state solution to \eqref{ST00}. 
Applying the change of variables $y=\frac{x-p_i}{\eps}$ for $u_i$, the above formula becomes 
$${\begin{split}
 \I(u) 
          & \le \frac{\sum_{i=1}^k a_i^p } {\big[ \sum_{i=1}^ka_i^q\big]^{p/q}}\, 
           \frac{\int_{\R^n} [ |Dw|^p+w^p] } {\big[ \int_{\R^n} w^q\big]^{p/q}} +o(1)\\
          & \le \frac{\sum_{i=1}^k a_i^p } {\big[ \sum_{i=1}^ka_i^q\big]^{p/q}}\, S_0 +o(1) ,
\end{split}} $$
where by the exponential decay of $w$, $o(1)=O(e^{-a/\eps})$ for some constant $a>0$. 
By Lemma \ref{L3.1},  one has
$H({\bf a})\le k^{1-p/q}$ when $a_i$ are close to 1.
Hence, we obtain \eqref{es0a}.

\vskip5pt

\subsubsection{  The case $k=0, l>1$.} \label{S3.6.4}
In this case, $  S_{0, k, l} =\big(\frac l2\big)^{1-p/q}  S_0$. 
The proof is similar to that of \S \ref{S3.6.3}, and is therefore omitted.

\vskip5pt

\subsubsection{  The case $k\ge 1, l\ge 1$.} \label{S3.6.5}
Finally, we consider the general case $k\ge 1, l\ge 1$.

First, we prove \eqref{es0b} in this case.
As in \S \ref{S3.6.3}, we choose points ${\bf p}=(p_1, \cdots, p_k)\in\Om^k$ and
${\bf q}=(q_1, \cdots, q_l)\in (\pom)^l$, such that $B_{{r}}(p_i)$
and $B_{{r}}(q_j)$ are disjoint, and $d_{p_i}\ge{r}$ for all $1\le i\le k$.
Choose coefficients $a_i=b_j=1$ for all $1\le i\le k$ and $1\le j\le l$.
By the monotonicity Lemma \ref{L3.2}, 
it suffices to consider the functional $\I$ on the domain 
$$\Om=B_{{r}} (p_1)\cup \cdots \cup B_{{r}} (p_k)\cup G_{r}(q_1)\cup\cdots\cup G_{r}(q_l),$$
where $G_r(q)=B_r(q)\cap\Om$. 
The peak function $u$ in the subsets $B_{r}(p_i)$ and $G_{r}(q_j)$ are completely independent. 
Hence, \eqref{es0b} follows from the arguments in  \S \ref{S3.6.1} and \S \ref{S3.6.2}. %{\it 3.6.1} and  {\it 3.6.2}.

Next, we prove \eqref{es0a}.
Choose arbitrary points $p_1, \cdots , p_k\in\Om$ and $q_1, \cdots, q_l\in\pom$ 
satisfying the $\delta$-apart condition.
Then the subsets $B_{\delta} (p_1), \cdots, B_{\delta} (p_k)$ and $G_{\delta}(q_1), \cdots, G_{\delta} (q_l)$
are disjoint and all contained in $\Om$.

Following the argument in \S \ref{S3.6.3}, we have the inequality
$${\begin{split}
\eps^{n(1-p/q)} \I(u) 
          \le \frac{\sum_{i=1}^k\int_\Om [\eps^p |Du_i|^p+u_i^p] +\sum_{j=1}^l \int_\Om [\eps^p |D\hat u_j|^p+\hat u_j^p]} 
          {\big[ \sum_{i=1}^k\int_\Om u_i^q + \sum_{j=1}^l \int_\Om \hat u_j^q\big]^{p/q}},
\end{split}} $$
where $u_i$,  for $i=1, \cdots, k$, are the same as above, 
and $\hat u_j$, for $j=1, \cdots, l$, are peak functions with a single peak at $q_j\in\pom$.

We then choose $u_i=a_iw\big(\frac{x-p_i}{\eps}\big)$ and $\hat u_j=b_j w\big(\frac{x-q_j}{\eps}\big)$.
Applying the change of variables $y=\frac{x-p_i}{\eps}$ for $u_i$ and  $y=\frac{x-q_j}{\eps}$ for $\hat u_j$, 
the above inequality becomes 
\beq\label{I343}
{\begin{split}
 \I(u) 
          & \le \frac{\sum_{i=1}^k a_i^p +\frac 12\sum_{j=1}^l b_i^p } {\big[ \sum_{i=1}^ka_i^q + \frac 12 \sum_{j=1}^l b_j^q\big]^{p/q}}\, 
           \frac{\int_{\R^n} [ |Dw|^p+w^p] } {\big[ \int_{\R^n} w^q\big]^{p/q}} +o(1)\\
          & \le \frac{\sum_{i=1}^k a_i^p +\frac 12\sum_{j=1}^l b_i^p } {\big[ \sum_{i=1}^ka_i^q + \frac 12\sum_{j=1}^l b_j^q\big]^{p/q}}
          \, S_0 +o(1).
\end{split}} \eeq
By Lemma \ref{L3.1}, we obtain  
$\I(u) \le  \big(k+{\Small\text{$\frac 12$}} l\big)^{1-p/q}S_0$, 
which implies that \eqref{es0a}.

\vskip10pt

\subsection{Estimate of $\I(w_{p,\eps})$}\label{S3.7}

Let $w$ be the ground state solution to \eqref{ST00} and denote $w_{p, \eps}=w(\frac{x-p}{\eps})$.

\begin{lemma}\label{L3.4}
Assume $d_p=N\eps$. There is a positive constant $c_N$, independent of $\eps$, such that
\beq\label{Es8}
\I(w_{p, \eps})\le S_0-c_N.
\eeq
\end{lemma}

\begin{proof}
Setting $y=\frac{x-p}{\eps}$,
we obtain 
\beq\label{a1}
{\begin{split}
\I(w_{p, \eps}) 
  &= \frac{\int_{\R^n\cap \{x_n>-N\} } [|Dw|^p+w^p]}{\big[\int_{\R^n\cap \{x_n>-N\} }w^q\big]^{p/q} } +o(1)\\
  &\le  \J(w) - C \int_{\R^n\cap \{x_n<-N\} } [|Dw|^p+w^p]+o(1),
  \end{split}}
\eeq
where $C>0$ and $o(1)\to 0$ as $\eps\to 0$ (independent of $N$).
Hence,  \eqref{Es8} follows.
\end{proof}

\begin{lemma}\label{L3.5}
Let $\bar{w} =\sum_{i=1}^{k} w_{p_i, \eps}+\sum_{j=1}^{l} w_{q_j, \eps}$. Then
\beq\label{Es9}
\I(\bar{w} )\le (k+l/2)^{1-p/q} S_0-c_N
\eeq
provided that the $\delta$-apart condition is violated for $\delta=N\eps$, 
that is, if one of the following conditions holds:
\begin{itemize}
\item [(i)] There exists a point $p_i$ in ${\bf p}=(p_1, \cdots, p_k)$ such that $d_{p_i}\le N\eps$;
\item [(ii)] There exist $i\ne j$ such that $|p_i-p_j|\le N\eps$;
\item [(iii)] There exist $i\ne j$ such that $|q_i-q_j|\le N\eps$.
\end{itemize}

\end{lemma}

\begin{proof}
In case (i),  assume that $d_{p_k}\leq N\eps$.
By the monotonicity Lemma \ref{L3.2},  we have
$${\begin{split}
\I(\bar{w} ) &\le \frac{\big[ (k-1)\int_{\R^n} +\int_{\R^n\cap\{x_n>-N\}} +l\int_{\R^{n,+}}\big] (|Dw|^p+w^p)+o(1)}
                      {\big[\big( (k-1)\int_{\R^n} +\int_{\R^n\cap\{x_n>-N\}} +l\int_{\R^{n,+}}\big)w^q +o(1)\big]^{p/q}} \\
       &\le \frac{\big[ k\int_{\R^n}  +l\int_{\R^{n,+}}\big] (|Dw|^p+w^p)}
                      {\big[\big( k\int_{\R^n}   +l\int_{\R^{n,+}}\big)w^q \big]^{p/q}} -c_N+o(1)\\               
       & = (k+l/2)^{1-p/q} S_0-c_N+o(1),
       \end{split}}$$
where $o(1)\to 0$ as $\eps\to 0$. Thus, we obtain \eqref{Es9}.

In case (ii) and (iii), the proof is similar, by properly dividing the domain $\Om$ and using the monotonicity Lemma \ref{L3.2}.
\end{proof}
 
 \vskip20pt

\section{A gradient flow}\label{S4}

To find critical points of the functional $\I$, we restrict $\I$ to peak functions,
which is a subset of the Sobolev space $W^{1,p}(\Om)$.
As a result, the standard variational theory cannot be applied directly. 
%we cannot apply the standard variational theory directly.
To overcome this difficulty, we employ a gradient flow to find critical points of $\I$.
This gradient flow is given by a parabolic $p$-Laplace equation.

Gradient flows are frequently used in the variational theory of fully nonlinear equations. 
For example, they have been applied to the Monge-Amp\`ere equation \cite{Chou} and the 
$k$-Hessian equation \cite{CW01}, 
where the corresponding equations are elliptic only 
under the condition that the functions are convex \cite{Chou} or admissible \cite{CW01}.

%In the variational theory for fully nonlinear equations, 
%one often uses gradient flows, 
%such as for the Monge-Ampere equation \cite{Chou} and  the $k$-Hessian equation \cite{CW01}, 
%where the equation is elliptic only when the function is convex \cite{Chou} or  admissible \cite{CW01}.

\subsection{The gradient flow}\label{S4.1}
Let
\beq\label{Is}
\I_s(u) =  \frac{\eps^{-n} \int_{\Om} (\eps^p(s+|Du|^2)^{p/2} +|u|^p)dx }{ \big(\eps^{-n}\int_{\Om} |u|^q\, dx\big)^{p/q}} ,
\eeq
where $s>0$ is a parameter. 
The reason to consider the functional $\I_s$ for $s>0$ is motivated by the regularity of the parabolic $p$-Laplace equation.
When $s=0$, the associated parabolic $p$-Laplace equation does not have $C^{2+\alpha, 1+\alpha}_{x, t}$ regularity,
which is essential when taking derivative in \eqref{dd} below.
For our purpose, we choose $s$ sufficiently small, such as $s\le e^{-1/\eps^4}$.

For any integers $k\ge 0, l\ge 0$ ($k+l\ge 1$), as in \S\ref{S3.1}, we introduce
\beq\label{Ss}
S_s({\bf p, q},  {\bf a, b})=\inf_{u\in \Phi_{k,l} ({\bf p, q},  {\bf a, b})} \,\I_s(u) 
\eeq
and 
\beq\label{Ss*}
S_s^*=\sup_{({\bf p,q})\in \Om^k\times \pom^l, \, ({\bf a,b})\in M} \ S_s({\bf p, q},  {\bf a,b}),
\eeq
where
${\bf p, q}$ satisfy the $\delta$-apart condition.
Similarly to Theorem \ref{T3.1}, we can prove the following result.

\begin{theorem}\label{T4.1}
For sufficiently small $\eps>0$ and  $0<s<e^{-1/\eps^4}$, 
the minimax $S_s^*$ is a critical value of the functional $\I_s$, 
and the corresponding critical point is a solution to
\beq\label{NPs}
{\begin{split}
-\eps^p \text{div}\big( (s+|Du|^2)^{\frac{p-2}{2}} Du\big)  & = \lambda u^{q-1} - u^{p-1}\ \ \text{in}\ \Om,\\
  u_\nu & =0\ \ \text{on}\ \pom, \\
   u& > 0\ \ \text{in}\ \Om
  \end{split}}
\eeq
with $k$ interior peaks and $l$ boundary peaks,
where $\lambda>0$ is a constant.
\end{theorem}

Let $y=x/\eps$ and denote $\Om^\eps=\Om/\eps$. 
Then, the functional $\I_s$ becomes
$$
\I_s(u) =  \frac{ \int_{\Om^\eps} ((\bar s+|Du|^2)^{p/2} +|u|^p)dy }{ \big(\int_{\Om^\eps} |u|^q\, dy\big)^{p/q}} ,
$$
where 
$$\bar s=\eps^{2} s .$$
We note that the estimates in Section \ref{S3} remain valid for $\I_s$,
since $\bar s=\eps^{2}s \to 0$ as $\eps\to 0$.

To derive the a priori estimates \eqref{uq}, we introduce the functional 
$$
\I_{s,\eta} [u] = \frac{ \eta\big[ 
   \int_{\Om^\eps} ((\bar s+|Du|^2)^{p/2} +|u|^p)dy\big] }{ \big(\int_{\Om^\eps} |u|^q\, dy\big)^{p/q}} . 
$$
We can choose a $C^{1,1}$ smooth function $\eta$, given by
\beq\label{eta1}
\eta(t)=\left\{ {\begin{split}
 & 2\alpha e^{\frac{t}{2\alpha}-1}  \ \ \ \ t>2\alpha,\\
 & t\ \ \hskip41pt  {\Small \text{$ \frac \alpha 2$}}  <t<2\alpha,\\
 & {\Small \text{$\frac{\alpha}{2}$}} e^{\frac{2t}{\alpha}-1}\ \ \ \ \ t< {\Small \text{$\frac \alpha 2$}},
 \end{split}} \right.
\eeq
where 
$${\begin{split}
 \alpha & = (k+ {\Small\text{$\frac l2$}} )\int_{\R^n} \big( |Dw|^p+w^p\big)dx.
 \end{split}} $$
 For brevity, we denote 
$${\begin{split}
A &=\int_{\Om^\eps} ((\bar s+|Du|^2)^{p/2} +|u|^p) dy,\\
B &=\int_{\Om^\eps} |u|^q\, dy.
\end{split}}$$
Let $u(x, t)$ be a solution to \eqref{PE0} subject to the Neumann condition $u_\gamma=0$ on $\pom^\eps$. 
We compute
\beq\label{dd}
{\begin{split}
\frac{d}{dt} \I_{s,\eta} (u(\cdot, t))
 &= \frac{p\eta'}{B^{p/q} } \int_{\Om^\eps} \big ((\bar s+|Du|^2)^{\frac {p-2}{2}} Du Du_t + u^{p-1} u_t \big)  
                         -  \frac{p\eta }{B^{p/q+1}}  \int_{\Om^\eps} u^{q-1}u_t \\
 &= \frac{p\eta'}{B^{p/q} } \Big[ \int_{\Om^\eps}  \big ((\bar s+|Du|^2)^{\frac {p-2}{2}}Du Du_t + u^{p-1} u_t \big) 
                         -  \int_{\Om^\eps} \frac{\eta}{\eta' B} u^{q-1}u_t\,  \Big]\\
 &= \frac{p\eta'}{B^{p/q} }  \int_{\Om^\eps}  \big( -\Delta_{p,\bar s} [u] +u^{p-1}-\frac{\eta}{\eta' B} u^{q-1} \big) u_t dy  ,
\end{split}} \eeq
where 
$$\Delta_{p, \bar s}[u] =\text{div} ((\bar s+|Du|^2)^{\frac {p-2}{2}}Du). $$ 
The integration by parts in \eqref{dd} requires the regularity of $u(\cdot, t)$, 
which is ensured by taking $s>0$.

We also point out that in the gradient flow, we only consider positive solutions. 
Therefore, we can drop the absolute value notation.

We introduce the parabolic equation
\beq\label{PE0}
{\begin{split}
 (u_t -\Delta_{p, \bar s}u) (y, t) &=  {\Small\text{$\frac{\eta}{\eta' B} $}} u^{q-1} - u^{p-1} , \ \ \ y\in\Om^\eps, t>0.
     \end{split}}
 \eeq
It follows from \eqref{eta1} and \eqref{dd} that 
 \beq\label{dgf}
\frac{d}{dt} \I_{s,\eta} (u(\cdot, t))=
   - \frac{p\eta'}{B^{p/q} }    \int_{\Om^\eps}  \Big( -\Delta_{p, \bar s} [u] +u^{p-1}-\frac{\eta}{\eta' B} u^{q-1} \Big)^2   dy \le 0. 
\eeq
The equality holds if and only if $u$ is a solution to \eqref{NPs}. 
From \eqref{dgf}, one sees that
\eqref{PE0} is a descent flow for the functional $\I_{s, \eta} $. 
%We choose a $C^{1,1}$ smooth function $\eta$, given by
%\beq\label{eta1}
%\eta(t)=\left\{ {\begin{split}
% & 2\alpha e^{\frac{t}{2\alpha}-1}  \ \ \ \ t>2\alpha,\\
% & t\ \ \hskip41pt  {\Small \text{$ \frac \alpha 2$}}  <t<2\alpha,\\
% & {\Small \text{$\frac{\alpha}{2}$}} e^{\frac{2t}{\alpha}-1}\ \ \ \ \ t< {\Small \text{$\frac \alpha 2$}}.
% \end{split}} \right.
%\eeq
%where 
%$${\begin{split}
% \alpha & = (k+ {\Small\text{$\frac l2$}} )\int_{\R^n} [ |Dw|^p+w^p] .
% \end{split}} $$

We define the minimax for the functional $\I_{s, \eta}$ analogously to \eqref{Ss} and \eqref{Ss*} for $\I_s$.
That is, we set
\beq\label{Seta}
{\begin{split} 
 & S_{s,\eta} ({\bf p, q},  {\bf a, b})  =\inf_{u\in \Phi_{k,l} ({\bf p, q},  {\bf a, b})} \,\I_{s,\eta}(u) ,\\
 & S_{s,\eta}^*=\sup_{({\bf p,q})\in \Om^k\times \pom^l, \, ({\bf a,b})\in M} \ S_{s,\eta} ({\bf p, q},  {\bf a,b}), 
 \end{split}}
\eeq
where ${\bf p, q}$ satisfy the $\delta$-apart condition.
Due to our choice of $\eta$, it follows that $S_{s,\eta}^*= S_s^*$.

Changing the coordinates $y$ in \eqref{PE0} back to $x$,  the equation \eqref{PE0} becomes
\beq\label{PE}
 (u_t - \eps^p \Delta_{p,s} u) (x, t)=   {\Small\text{$\frac{\eta}{\eta' B} $}}u^{q-1} - u^{p-1}  \ \ \ \text{in}\ Q_T,
 \eeq
 where $\Delta_{p,  s}[u] =\text{div} ((s+|Du|^2)^{\frac {p-2}{2}}Du)$ and $Q_T=\Om\times (0, T]$. 
 The initial boundary condition for \eqref{PE} is 
\beq\label{PEb}
{\begin{split}
u(\cdot, 0) &= \phi_0  \ \ \text{on}\ \Om\times \{t=0\} ,\\
u_\nu  & =0\ \ \ \ \text{on}\ \pom\times (0, T].
\end{split}}
\eeq
We will use \eqref{PE}-\eqref{PEb} to study the critical points of the functional $\I$.

 %%A different gradient flow, also from [JFA2010]. Let $v(\cdot, t)=\frac{u(\cdot, t)} {\|u\|}$.  Then
 %%$${\begin{split}
%%  & \frac{d}{dt}  \I( v(\cdot, t))=\frac{d}{dt} \I( u(\cdot, t))
 %%  =  \frac{d}{dt}  \frac{ \int_{\Om^\eps}  (|Du|^p +|u|^p)dy}{  \big(  \int_{\Om^\eps}  |u|^q\, dy\big)^{p/q}} \\
 %% &= \frac{p}{B^{p/q}} \int_{\Om^\eps} \big[ |Du|^{p-2}Du Du_t+u^{p-1} u_t\big] -\frac{pA}{B^{p/q+1} }\int_{\Om^\eps} u^{q-1} u_t\\
%%&=\frac{p\|u\|^{p-1}}{B^{p/q}} \int_{\Om^\eps} \Big[ -\Delta_p v +v^{p-1}  -\frac AB v^{q-1} \Big] u_t
%% \end{split}} $$

 \vskip5pt

\subsection{A priori estimates for the gradient flow}\label{S4.2}
Next, we recall the a priori estimates for the gradient flow \eqref{PE}-\eqref{PEb}.
First, we have

\begin{lemma}\label{L4.1}
There exist positive constants $\bar b_1, \bar b_2$, depending only on ${\bf p}, { \bf q}, n, \eta, \Om$ and the initial condition $\phi_0$, 
but independent of $s\in [0, 1]$, such that the solution $u=u_{s, \eta}$ to \eqref{PE}-\eqref{PEb} satisfies 
\beq\label{uq}
\bar b_1\leq \int_{\Om^\eps} |u|^q\, dy \leq \bar b_2. 
\eeq
\end{lemma}

\begin{proof}
For any initial condition $\phi_0$, since \eqref{PE}-\eqref{PEb} is a descent flow, there holds 
$ \I_{s, \eta} (u(\cdot, t))\le \I_{s, \eta} (\phi_0)$, i.e.,
$\frac{\eta(A)}{B^{p/q}}\le  \I_{s, \eta} (\phi_0)$, where the notations $A, B$ are the same as in  \eqref{dd}.
Hence,
$$B^{p/q} \ge \frac {\eta(A) }{\I_{s, \eta} (\phi_0)}\ge \frac{\alpha/e}{2\I_{s, \eta} (\phi_0)} =: (\bar  b_1)^{p/q}. $$

On the other hand, 
since $\eta$ we have chosen is an exponential function when $t>2\alpha$, by the Sobolev embedding, we get
$$\I_{s, \eta}(\phi_0) \ge \frac{\eta(A)}{B^{p/q}} \ge c_1 \frac{A^2}{B^{p/q}}  \ge c_2 B^{p/q},$$
which implies that $B^{p/q}\le c_2^{-1} \I_{s, \eta}(\phi_0)$. 
\end{proof}

%{\color{blue}
Lemma \ref{L4.1}, together with Lemma \ref{L4.4}, implies the $L^\infty$ 
estimate for the solution to our modified gradient flow, namely, the flow described in \S\ref{S4.3} below.
%}
 
For the local existence of solutions to \eqref{PE}-\eqref{PEb}, we refer the reader to \cite[Lemma 5.1]{TW10},
where the local existence was proved for equations with integrals on the RHS, such as $A, B$ in \eqref{PE}. 
Note that for $s>0$, equation \eqref{PE} is uniformly parabolic. 

With the local existence of solutions established,  the a priori estimates in Lemma \ref{L4.1} 
and the regularity results Theorems \ref{T4.2}-\ref{T4.3}
yield the long time existence of solutions. 
By sending $s\to 0$, one obtains the existence of weak solutions to \eqref{PE}-\eqref{PEb}
in the case $s=0$.
 
Next, we recall the following regularity result for solutions to \eqref{PE}–\eqref{PEb}, as established in \cite{DB93}.
%we quote the following regularity of solutions to \eqref{PE}-\eqref{PEb}  \cite{DB93}.
Theorems \ref{T4.2}-\ref{T4.3} and the a priori estimates all depend on the initial condition. 
For application in this paper, we choose the initial function given in \eqref{weps},
which is a very nice function in the coordinates $y=x/\eps$.

\begin{theorem}\label{T4.2}
For any $T>0$,
the gradient flow \eqref{PE}-\eqref{PEb} admits a unique solution $u=u_{s,\eta}\in C^\gamma(\overline Q_T)$.
Moreover,  the solution satisfies the estimate
\beq\label{eM}
\|u(\eps y, t)\|_{C^\gamma ({\overline Q}^\eps_T)}\le M,
\eeq
where $Q^\eps_T=\Om^\eps\times (0, T]$, $\Om^\eps=\Om/\eps$, and  the constant $M>0$ is independent of  $\eps, s$.
\end{theorem}

In estimate \eqref{eM}, we need to make the dilation $y=x/\eps$, due to the coefficient $\eps$ 
in equation \eqref{PE}.  The proof of Theorem \ref{T4.2} is local.
Therefore, after the change of variables $y=x/\eps$, 
we obtain local estimates for $u$ in the coordinates $y$.

Moreover, the H\"older continuity of $D_xu$ for   equation \eqref{PE} has also been established
(see \cite{DBF85, IJS17}).

\begin{theorem}\label{T4.3}
 Let $u=u_{s, \eta}$ be the solution to \eqref{PE}-\eqref{PEb}.
Then the following estimate holds
\beq\label{eM1}
\|D_yu(\eps y, t)\|_{C^\gamma ({\overline Q}^\eps_T)}\le M,
\eeq
where the constant $M>0$ is uniform for small $\eps, s>0$.
\end{theorem}

%% Recently, new estimates for second derivatives have been obtained \cite{FPS22, AS24}. 
%% It was proved that the solution $u$ to \eqref{PE}-\eqref{PEb} 
%% has the regularity $u_t\in L^2(Q_T)$ and  $|Du|^{p-2}Du\in L^2([0, T],\  W^{1,2}(\Om))$.

Very recently, the boundedness of the time derivative $u_t$ was obtained in \cite{LLYZ}.

Our choice of $\eta$ in \eqref{eta1} is aimed at establishing the a priori estimate \eqref{uq}.
For $\frac {\alpha}{2}<t<2\alpha$, we have $\eta(t)=t$.
In our proof of Theorem \ref{T3.1}, we look for a solution $u$  satisfying
$$\eps^{-n}\int_\Om [\eps^p |Du|^p+|u|^p] dx=\int_{\Om^\eps} [ |Du|^p+|u|^p] dy
            \to \alpha\ \ \text{as}\ \eps\to 0,$$
where $\alpha  = (k+ {\Small\text{$\frac l2$}} )\int_{\R^n} ( |Dw|^p+w^p)dx$,
as given in \eqref{eta1}.
  
\vskip10pt

\subsection{Freezing the gradient flow \eqref{PE}-\eqref{PEb}}\label{S4.3}
We apply the descent gradient flow to find a critical point of the functional $\I_{s,\eta}$ with critical value $S_s^*$.
Let $u(x, t)$ be a solution to \eqref{PE}-\eqref{PEb}.
If there is a time $t^*\ge 0$ such that $\I_{s, \eta}(u(\cdot, t^*))= S_s^*-\sigma$ 
(for a given but arbitrarily small constant $\sigma>0$), 
then  the monotonicity of the flow implies 
$\I_{s,\eta} (u(\cdot, t))\le S_s^*-\sigma$ for all time $t \ge t^*$.
Thus, the solution $u(\cdot, t)$ becomes irrelevant to the critical value  $S_s^*$ for $t>t^*$.
For convenience of our argument, we freeze the solution for $t>t^*$. \ \ \

Accordingly, in the following discussion, a solution $u(\cdot, t)$ to \eqref{PE}-\eqref{PEb} will always satisfy
$\I_{s,\eta}(u(\cdot, t) )> S_{s}^*-\sigma$,
and if there is a time $t^*\ge 0$ such that $\I_{s,\eta} (u(\cdot, t^*) )=S_s^*-\sigma$, 
then we set $\I_{s,\eta} (u(\cdot, t) ) =S_s^*-\sigma$ for all $t>t^*$.
For clarity, we choose 
$$\sigma=e^{-1/\eps^2} . $$

\subsection{Solution to \eqref{PE}-\eqref{PEb} is a peak function.}\label{S4.4}
To investigate the properties of solutions to the parabolic $p$-Laplace equation \eqref{PE}-\eqref{PEb}, we first establish the following convergence lemma.
%We first show that 
%under appropriate conditions, 
%the functional $\I_{s, \eta}$ is strictly monotone, in the sense \eqref{md1}.

\begin{lemma}\label{L4.2} 
 If there exists a sequence $t_j\to t_0$ such that $\I_{s,\eta}(u(\cdot, t_j) )> S_{s}^*-\sigma$
 and $\frac{d}{dt} \I_{s, \eta} (u(\cdot, t_j))\to 0$, 
 then the limit $u(\cdot, t_0)=\lim_{j\to\infty} u(\cdot, t_j)$ is a solution to 
 \beq\label{van}
 \Delta_{p, \bar s}  u(\cdot, t_0) -u^{p-1} (\cdot, t_0) + 
{\Small\text{$\frac{\eta}{\eta' B}$}} u^{q-1}(\cdot, t_0) =0 \ \ \text{in}\ \Om^\eps. 
\eeq
\end{lemma}

\begin{proof}
By \eqref{dgf},   $\frac{d}{dt}  \I_{s, \eta} (u(\cdot, t_j))\to 0$ implies that
$$\Delta_{p, \bar s}  u(\cdot, t_j) -u^{p-1} (\cdot, t_j) + 
{\Small\text{$\frac{\eta}{\eta' B}$}} u^{q-1}(\cdot, t_j) =f(\cdot, t_j) \in L^2(\Om^\eps)$$
and $\|f(\cdot, t_j)\|_{L^2(\Om^\eps)}\to 0$. 
Hence, the limit $u(\cdot, t_0)$ is a solution of \eqref{van}.
\end{proof}

The condition $\I_{s,\eta}(u(\cdot, t_j) )> S_{s}^*-\sigma$ in Lemma \ref{L4.2} ensures that the solution is not frozen. 
By the estimates in  \cite{LLYZ}, we have $f(\cdot, t_j) \in L^\infty(\Om^\eps)$.

\begin{lemma}\label{L4.3} 
Let $u(x, t)$ be a solution to \eqref{PE}-\eqref{PEb} satisfying $\I_{s,\eta}(u(\cdot, t) )> S_{s}^*-\sigma$.
Then, we have the estimate
$$
\int_{\omega} \big[ u(y, t_1) -   u(y, t_0)\big] dy \le C  |t_1-t_0|^{1/2}
\big| \I_{s, \eta} (u(\cdot, t_1))- \I_{s, \eta} (u(\cdot, t_0)) \big|^{1/2} 
$$
for any sub-domain $\omega\subset \Om^\eps$ with bounded volume $|\omega|\le C$. 
\end{lemma}

\begin{proof}
From the H\"older inequality and \eqref{dgf}, 
we calculate
$${\begin{split}
\frac{d}{dt} \int_{{\omega } } u & =\int_{{\omega } } u_t\le |{\omega } |^{1/2} \big[\int_{{\omega } } u^2_t\big]^{1/2}\\
 & \le |{{\omega } }|^{1/2} \big[\int_{\Om^\eps} u^2_t\big]^{1/2}
\le C \Big|\frac{d}{dt}\I_{s, \eta} (u(\cdot, t))\Big|^{1/2}.
\end{split}}
$$
Hence, one has
$${\begin{split}
\int_{{\omega } } u(\cdot, t_1) - \int_{{\omega } } u(\cdot, t_0)
 &= \int_{t_0}^{t_1} \frac{d}{dt} \int_{{\omega } } u 
\le C\int_{t_0}^{t_1} \Big| \frac{d}{dt} \I_{s, \eta} (u(\cdot, t))\Big|^{1/2}\\
 &\le  C  |t_1-t_0|^{1/2} \Big| \int_{t_0}^{t_1} \frac{d}{dt} \I_{s, \eta} (u(\cdot, t))\Big|^{1/2}\\
 &\le  C  |t_1-t_0|^{1/2} \big| \I_{s, \eta} (u(\cdot, t_1))- \I_{s, \eta} (u(\cdot, t_0)) \big|^{1/2} . 
\end{split}}$$
\vskip-25pt
\end{proof}

Applying Lemma \ref{L4.3} to $\{x\in\omega\ |\ u(\cdot, t_1) >   u(\cdot, t_0)\}$ and 
$\{x\in\omega\ |\ u(\cdot, t_1) <   u(\cdot, t_0)\}$ respectively, we obtain the following estimate.

\begin{corollary}\label{C4.1} 
Let $u(x, t)$ be a solution to \eqref{PE}-\eqref{PEb}. Then  
\beq \label{uIs}
\int_{\omega} \big| u(y, t_1) -   u(y, t_0)\big| \, dy \le C  |t_1-t_0|^{1/2}
\big| \I_{s, \eta} (u(\cdot, t_1))- \I_{s, \eta} (u(\cdot, t_0)) \big| ^{1/2} 
\eeq
for any sub-domain $\omega\subset \Om^\eps$ with bounded volume $|\omega|\le C$. 
\end{corollary}

We then show that under appropriate conditions, 
the functional $\I_{s, \eta}$ is strictly monotone, in the sense \eqref{md1}. 
Define
\beq\label{weps1}
\phi_\Lambda ={\Small\text{$ \sum_{i=1}^k$}}  \, a_iw\big( {\Small\text{$\frac{x-p_i}{\eps}$}} \big)
             + {\Small\text{$\sum_{j=1}^l$}} \, b_j w\big({\Small\text{$\frac{x-q_j}{\eps}$}} \big) ,
\eeq
where {\bf p, q} satisfy the $\delta$-apart condition, and the coefficients
\beq\label{ab1}
a_1, \cdots, a_k, b_1, \cdots, b_l\in (1-\hat \delta, 1+\hat\delta) .
\eeq
The subscript $\Lambda $ denotes 
\beq \label{Lam}
 \Lambda=   (p_1, \cdots, p_k, q_1, \cdots, q_l, a_1, \cdots, a_{k}, b_1, \cdots, b_{l} ).
\eeq

Let $u(x, t)=u_\Lambda(x, t)$ be the solution to \eqref{PE}-\eqref{PEb} with initial condition $\phi_\Lambda$.
At the beginning, when $t>0$ is sufficiently small, the solution $u(\cdot, t)$ is a peak function,
which has interior local mass centres 
${\bf p}_{\Lambda, t}=(p_{1,t}, \cdots, p_{k, t})$, 
boundary local mass centres ${\bf q}_{\Lambda, t}=(q_{1, t}, \cdots, q_{l, t})$ by \S \ref{S2.4}, 
and has the coefficients ${\bf a}_{\Lambda, t}=(a_{1, t}, \cdots, a_{k, t})$,
${\bf b}_{\Lambda, t}=(b_{1, t}, \cdots, b_{l, t})$ by \S \ref{S2.3}. 
%{\color{blue}
Denote 
\beq\label{Lamt}
{\begin{split}
\Lambda_t  =({\bf p}_{\Lambda, t}, {\bf q}_{\Lambda, t}, {\bf a}_{\Lambda, t}, {\bf b}_{\Lambda, t}),\ \
\Lambda'_t =({\bf p}_{\Lambda, t}, {\bf q}_{\Lambda, t}, {\bf a'}_{\Lambda, t}, {\bf b'}_{\Lambda, t}),\
\end{split}}
\eeq
where ${\bf a'}_{\Lambda, t}=(a'_{1, t}, \cdots, a'_{k, t})$,
 ${\bf b'}_{\Lambda, t}=(b'_{1, t}, \cdots, b'_{l, t})$, 
\beq\label{ab1a}
a'_{i, t}=:\frac{a_{i, t}}{\bar a_t} ,\ \  b'_{j, t}=:\frac {b_{j, t}}{\bar a_t}
\ \ \text{for}\ \ i=1, \cdots, k,\ j=1, \cdots, l,
\eeq
and $\bar a_t ={\Small\text{$\frac{1}{k+l}$}} (a_{1, t} +\cdots + a_{k, t} +b_{1, t}+ \cdots +b_{l, t}). $
The average value $\bar a_t$ in \eqref{ab1a} is due to the homogeneity of the functional $\I$.
The functional $\I$ is invariant under multiplication by any positive constant.
We also denote 
$$
\phi_{\Lambda_t} ={\Small\text{$ \sum_{i=1}^k$}}  \, a_{i, t}w\big( {\Small\text{$\frac{x-p_{i,t}}{\eps}$}} \big)
             + {\Small\text{$\sum_{j=1}^l$}} \, b_{j, t} w\big({\Small\text{$\frac{x-q_{j, t}}{\eps}$}} \big) . 
$$
As $t$ increases, from the solution $u(\cdot, t)$ we can define a unique peak function 
$\phi_{\Lambda_t}$ as long as  ${\bf p}_{\Lambda, t}, {\bf q}_{\Lambda, t}$ satisfy the $\delta$-apart condition,
and the coefficients ${\bf a'}_{\Lambda, t}, {\bf b'}_{\Lambda, t}$ satisfy 
\beq\label{ab1b}
a'_{i, t}, b'_{j, t} \in (1-\hat \delta, 1+\hat\delta)
\eeq
for $i=1, \cdots, k, \ j=1, \cdots, l$.
%}

 \begin{lemma}\label{L4.4} 
Let $u(x, t)$ be a solution to \eqref{PE}-\eqref{PEb}
with initial condition $\phi_\Lambda$. Assume 
\beq\label{Ssg}
\I_{s, \eta} (u(\cdot, t))>S_s^*-\sigma . 
\eeq
Then 
\beq\label{md1}
\frac{d}{dt} \I_{s, \eta} (u(\cdot, t))<-\delta_1
\eeq
for some constant $\delta_1>0$, 
 if one of the following conditions holds (assume that $k+l>1$ in cases {\bf c)} - {\bf e)} below):
\begin{itemize}
\item [{\bf a})] $ \frac 12\bar\delta\le \|u(\cdot, t)-\phi_{\Lambda_t}\|_{L^\infty(\Om)} \le \bar\delta$;  
\item [ {\bf b})] There exists $1\le i\le k$ such that $N\eps \le d_{p_{i,t}}  \le  2N\eps$;
\item [ {\bf c})] There exist $i\ne j$, $1\le i, j\le k$, such that $N\eps \le |p_{i,t}-p_{j,t}|  \le 2N\eps$;
\item [ {\bf d})] There exist $i\ne j$, $1\le i, j\le l$, such that $N\eps \le |q_{i,t}-q_{j,t}| \le 2N\eps$;
\item [ {\bf e})] %{\color{blue} 
          There exists $1\le i\le k$ such that 
          $\frac 12\hat\delta \le |a'_{i,t}-1|  \le  \hat\delta$;
          or there exists $1\le j\le l$ such that  
          $\frac 12 \hat\delta \le | b'_{j, t}-1| \le \hat\delta$. %}
\end{itemize}
\end{lemma}

%% \noindent{\bf Remark.}\ \ 
%% \eqref{md1} should be understood in the following sense,
%% \beq\label{md2}  \I_{s, \eta} (u(\cdot, t^*))-\I_{s, \eta}(u(\cdot, t')<-\delta_1(t^*-t')  \eeq
%% for $t^*>t'$ and $t^*-t'=\bar \sigma$, for a sufficiently small constant $\bar \sigma>0$.
%% The proof is uniform for $s>0$, but after taking limit, we are at $s=0$. 

\begin{proof} 
\vskip-4pt
For the initial condition $\phi_\Lambda$,  none of Cases {\bf a}) - {\bf e}) in Lemma \ref{L4.4} occurs at $t=0$.
Thus, they do not occur for small $t>0$.
Suppose that one of {\bf a}) - {\bf e})  first occurs at some time $t>0$.
We then prove that \eqref{md1} holds at this time $t$.

Suppose to the contrary that \eqref{md1} is not true.
%% namely when $\eps\to 0$, 
%% \beq\label{md3}  \I_{s, \eta} (u(\cdot, t^*))-\I_{s, \eta}(u(\cdot, t')\to 0 . \eeq
%% Note that $u$ and $t^*$ depend on $\eps$ and $s$. Let us ignore the dependence at the moment.  
We make the dilation $y=\frac{x-p}{\eps}$ for a proper $p$, and send $\eps\to 0$.
Then, by \eqref{peak2a} and  \eqref{eM1}, the solution $u(y, t)$  converges to a limit $u_0$, 
which is a solution to 
\beq\label{lamb}
-\Delta_p u=\lambda u^{q-1}-u^{p-1}
\eeq
for a positive constant $\lambda$.
Note that as $\eps\to 0$, $\bar s=\eps^{2} s\to 0$ for the constant $\bar s$ in \eqref{PE0}.
Without loss of generality, we may assume  $\lambda=1$ by a constant multiplication.
Hence, $u_0=w$ after a translation.

But if one of the conditions {\bf a}) - {\bf e}) holds, we will show below that $u_0\ne w$,
leading to a contradiction.
Therefore, \eqref{md1} must hold.
 
In the following, we show that $u_0\ne w$ case by case.

\begin{itemize}

\item [{\underline{Case {\bf a})}}.]  
Assume that Case {\bf a}) occurs first among Cases {\bf a}) to {\bf e}) at time $t$.
Then we have 
$\|u(\cdot, t)-\phi_{\Lambda_t}\|_{L^\infty(\Om)} = \theta \bar\delta$ for some $\theta\in [\frac 12, 1]$. 
Let $p_{\eps, t} \in\Om$ be a point such that $|u(p_{\eps, t}, t)-\phi_{\Lambda_t}(p_{\eps, t})| =\theta \bar\delta$.
As before,  let $y=\frac{x-p_{\eps, t}}{\eps}$ and $u_0(y)=\lim_{\eps\to 0} u (y, t)$.
Let $p_{i, t}$ denote a peak of  $\phi_{\Lambda_t}$ that is closest  to $p_{\eps, t}$ among all peaks of $\phi_{\Lambda_t}$.

If there is a constant $N'>0$ such that $|p_{i,t}-p_{\eps, t}|\le N'\eps$,
then $u_0$ is a function with a peak at $\bar p=:\lim_{\eps\to 0}  \frac{p_{i,t}-p_{\eps, t}}{\eps}$ with $|\bar p|\le N'$.
By condition {\bf a}), it follows that  $u_0$ and $w$ are separated by a positive distance, yielding a contradiction.

If $|p_{i,t}-p_{\eps,t}|/\eps\to\infty$,
then $u(0) = \theta \bar\delta$ (in the coordinates $y$ ) and $u(0)=\max_{\R^n} u$, which also leads to  a contradiction.

\vskip5pt

\item [{\underline{Case  {\bf b})}}.] 
Assume that Case {\bf b}) occurs first among Cases {\bf a}) to {\bf e}) at time $t$.
Let $p\in\pom$ be a point such that $d_{p_{i, t}}=|p_{i, t}-p|$.
Set $y=\frac{x-p}{\eps}$ and define $u_0(y)=\lim_{\eps\to 0} u (y, t)$.
Then the limit $u_0$, defined in $\R^n\cap\{x_n>0\}$, 
 satisfies the Neumann condition on $\{x_n=0\}$ 
and has a peak at the point $Re_n$ for some $R\in[N, 2N]$, % satisfying $N\le R\le 2N$, 
where $e_n$ is the unit vector on the $x_n$-axis.
By making the even extension in $\{x_n=0\}$, we obtain an entire solution
to \eqref{ST00} with at least two peaks, 
which is different from the ground state solution $w$,
thus contradicting the uniqueness of  solutions to \eqref{ST00}.
 
\vskip5pt

\item [{\underline{Case {\bf c})}}.] 
Assume that Case {\bf c}) occurs first among Cases {\bf a}) to {\bf e}) at time $t$.
Let  $y=\frac{x-p_{i, t}}{\eps}$ and $u_0(y)=\lim_{\eps\to 0} u (y, t)$.
Then $u_0$ is a function with two peaks, and  is defined either in the entire space 
or  in a half space, say $\{x_n>a\}$ for some constant $a$.

In the former case,  $u_0$ is a ground state solution to \eqref{ST00},
and so $u_0=w$, which contradicts the assumption that $u_0$ has two peaks.

In the latter case, $u_0$ satisfies the Neumann condition on $\{x_n=a\}$.
By performing the even extension, we obtain a ground state solution to \eqref{ST00} with more than one peak, 
again contradicting the uniqueness of the ground state solution. 
%We also reach a contradiction.

\vskip5pt

\item [{\underline{Case {\bf d})}}.] 
Choose the coordinates such that both $q_{i, t}$ and $q_{j, t}$
are located on the plane $\{x_n=0\}$.
Let  $y=\frac{x-q_{i, t}}{\eps}$ and $u_0(y)=\lim_{\eps\to 0} u (y, t)$.
Then $u_0$ is defined in the half space and
 satisfies the Neumann condition on $\{x_n=0\}$.
We reach a contradiction by making an even extension as above.

\vskip5pt

\item [{\underline{Case {\bf e})}}.] 
Let us consider the case when $\frac 12\hat\delta \le |a'_{i,t}-1|  \le  \hat\delta$, 
the other case can be thandled in the same  way.
By the definition of ${\overline a}_t$,
there is a coefficient among {\bf a} and {\bf b} (say $a_{j, t}$ or  $b_{j, t}$) such that 
 $|a_{i, t}-a_{j,t}|>\frac 12\hat\delta$ or $|a_{i, t}-b_{j,t}|>\frac 12\hat\delta$.

Let  $y=\frac{x-p_{i, t}}{\eps}$ and $u_0(y)=\lim_{\eps\to 0} u (y, t)$.
From the above argument, we must have $u_0=w$ (since we assume $\lambda=1$ in \eqref{lamb} by a constant multiplication), 
which implies that $a_{i, t}\to 1$ as $\eps\to 0$.
Similarly, we have  $a_{j, t}\to 1$ as $\eps\to 0$,
contradicting the condition $|a_{i, t}-a_{j,t}|>\frac 12\hat\delta$.
\end{itemize}

%{\color{blue} 
To complete the proof,
it suffices to notice that the above proof is valid 
if two or more cases among the conditions {\bf a}) - {\bf e}) occurs simultaneously. %}
\end{proof}

%% The proof implies that \eqref{md1} holds if the conditions in Lemma \ref{L4.4} is replaced by
%% \begin{itemize}
%% \item [ {\bf a$'$})] $\frac 14\bar\delta< \|u(\cdot, t)-\phi_{\Lambda_t}\|_{L^\infty(\Om)} <  \frac 12\bar\delta. $  
%% \item [ {\bf b$'$})] There exist $1\le i\le k$ such that $N\eps<d_{p_{i,t}} < 2N\eps$.
%% \item [ {\bf c$'$})] There exist $i\ne j$, $1\le i, j\le k$, such that $N\eps<|p_{i,t}-p_{j,t}| < 2N\eps$. 
%% \item [ {\bf d$'$})] There exist $i\ne j$, $1\le i, j\le l$, such that $N\eps<|q_{i,t}-q_{j,t}| < 2N\eps$.
%% \item [ {\bf e$'$})] There exist $1\le i\le k$ such that $\frac 14 \hat\delta < |a_{i,t}- {\overline a}_t| <\frac 12 \hat\delta$;
%%           or there exist $1\le j\le l$ such that  $\frac 14 \hat\delta < |b_{j, t}-{\overline a}_t|< \frac 12 \hat\delta$.
%%  \end{itemize}

Let $u(x, t)$ be a solution to \eqref{PE}-\eqref{PEb} with initial condition $\phi_\Lambda$,
where {\bf p, q} satisfy the $\delta$-apart condition and {\bf a,b} satisfy \eqref{ab1}. 
We say that the solution $u$ traverses condition {\bf a})  
if there exist times $\tau_1$ and $\tau_2$, $0\le \tau_1 < \tau_2\le \infty$, 
such that 
$${\begin{split}
 & {\Small\text{$\frac 12$}} \bar\delta \le \|u(\cdot, t)-\phi_{\Lambda_t}\|_{L^\infty(\Om)} \le \bar\delta\ \ \ \mbox{for any} \ t\in (\tau_1, \tau_2),\\
  & \|u(\cdot, \tau_1)-\phi_{\Lambda_{\tau_1}}\|_{L^\infty(\Om)} ={\Small\text{$\frac 12$}}   \bar\delta, \\
 & \|u(\cdot, \tau_2)-\phi_{\Lambda_{\tau_2}}\|_{L^\infty(\Om)} =  \bar\delta.
  \end{split}} $$
Similarly, corresponding terminology can be defined for the remaining conditions in Lemma \ref{L4.4}.

%we can define the terminology for other conditions in Lemma \ref{L4.4}.

\begin{lemma}\label{L4.5}
If the solution $u$ traverses one condition among the conditions {\bf a}) - {\bf e}), 
then there is a constant $\beta>0$, independent of $\eps>0$, such that 
\beq\label{utt}
\int_{B_1} \big| u(y, \tau_2) -   u(y, \tau_1)\big|\, dy \ge \beta  
\eeq
for a ball $B_1$ in $\Om^\eps$.
\end{lemma}
  
\begin{proof}
It suffices to observe that if the solution $u$ traverses any one condition among {\bf a}) - {\bf e}), 
then there exists a point $z\in\Om^\eps$ such that
$|u(z, \tau_1)-u(z, \tau_2)|>\beta_1$ for some constant $\beta_1>0$ independent of $\eps>0$.
Therefore, \eqref{utt} follows from the regularity of the solution.
\end{proof}

If the initial condition $\phi_\Lambda$ satisfies  $\I_{s, \eta}(\phi_\Lambda)<S^*_s-\sigma$,
then by \S \ref{S4.3}, the gradient flow is frozen from $t=0$. % and there is no need to consider the flow anymore.
Consequently, the behavior of the flow for $t>0$ is irrelevant and need not be analyzed further.
By Lemma \ref{L3.5} and the concavity property in \S \ref{S3.5}, we have $\I_{s, \eta}(\phi_\Lambda)<S^*_s-\sigma$
if one of conditions {\bf b}) - {\bf e}) in Lemma \ref{L4.4} holds.
The computations in Section \ref{S3} were done for the case $s=0$, but they remain valid for sufficiently small $s$.
Recall that we assume $0<s<e^{-1/\eps^4}$. 

If  $\I_{s, \eta}(\phi_\Lambda)>S^*_s-\sigma$, 
we show that the flow $u(\cdot, t)$ will be frozen before it can traverse any condition among {\bf a}) - {\bf e}).

\begin{lemma}\label{L4.6}
The flow $u(\cdot, t)$ cannot traverse any condition among {\bf a}) - {\bf e}) before it is frozen.
\end{lemma}

\begin{proof}
If the solution $u(\cdot, t)$ traverses a condition among {\bf a}) - {\bf e}) from time $\tau_1$ to time $\tau_2$,
we derive from Lemma \ref{L4.4} that \eqref{md1} holds for all $t\in (\tau_1, \tau_2)$.
%then by Lemma \ref{L4.4}, we have \eqref{md1} for all $t\in (\tau_1, \tau_2)$.

If the flow is not frozen, one obtains \eqref{Ssg}. 
By our computation in \S \ref{S3.6}, we have $\I_{s, \eta}(\phi_\Lambda)<S_s^*+o(1)$ as $\eps\to 0$,
which, together with  \eqref{Ssg}, implies
$$\big| \I_{s, \eta} (u(\cdot, \tau_2))- \I_{s, \eta} (u(\cdot, \tau_1)) \big| = o(1)\ \ \text{ as $\eps\to 0$}.  $$
Therefore, by \eqref{md1}, we get $\tau_2-\tau_1=o(1)$ as $\eps\to 0$. 
It then follows from \eqref{uIs}  that  the integral 
$$
\int_{\omega} \big| u(\cdot, \tau_2) -   u(\cdot, \tau_1)\big|=o(1)\ \ \ \text{as $\eps\to 0$},\ 
$$
for any sub-domain $\omega$ with bounded volume, which contradicts \eqref{utt}.
\end{proof}

Note that conditions {\bf a}) - {\bf e}) in Lemma \ref{L4.4}  
contain all the possibilities for the solution $u_\Lambda(\cdot, t)$  
to run out of the set of peak functions.
In particular, condition {\bf a}) also includes the case when a boundary peak moves into the interior.
That is, when a boundary peak moves inward, condition {\bf a)} occurs.

Therefore, we have 

\begin{corollary}\label{C4.2}
Let $u(\cdot, t)$ be the solution to the parabolic $p$-Laplace equation \eqref{PE}-\eqref{PEb}
with initial condition $\phi_\Lambda$. 
Then,  $u(\cdot, t)$ is a peak function if $\I_{s,\eta}(u(\cdot, t))>S_s^*-\sigma$.
\end{corollary}

\vskip10pt

\begin{corollary} \label{C4.3} 
Let $u(\cdot, t)$ be as above.
If  $\Lambda'_t\to \p G$, then  $\I_{s,\eta}(u (\cdot, t))\le S_s^*-\sigma$,
where   $\Lambda'_t=({\bf p}_{\Lambda, t}, {\bf q}_{\Lambda, t}, {\bf a'}_{\Lambda, t}, {\bf b'}_{\Lambda, t})$ is given in \eqref{Lamt}, and $G$ is  defined in \eqref{G}.
\end{corollary}

\subsection{Convergence of the gradient flow}\label{S4.5} 

Let $u(\cdot, t)=u_{s,\eta}(\cdot, t)$ be a solution to \eqref{PE}-\eqref{PEb}. 
If 
\beq\label{aim}
\I_{s, \eta}(u(\cdot, t))\ge S^*_s\ \ \text{for all $t>0$, }
\eeq
then by the regularity \eqref{eM1} and the estimate \eqref{dgf},
we conclude that $u(\cdot, t)$ converges to a limit $u_{0,s}$, which is a solution to \eqref{NPs}.

For the case $s=0$, note that $u_{0, s}$ is a peak function as in Definition \ref{D2.1},
and its peaks satisfy the $\delta$-apart condition.
Thus, by sending $s\to 0$,
we see that the limit $u_0=\lim_{s\to 0} u_{0, s}$ is a weak solution to \eqref{NP}, 
with a Lagrange multiplier that can be removed by a multiplication of the solution.
Hence, Theorem \ref{T3.1} follows.

Therefore, to prove Theorem  \ref{T3.1}, it suffices to find an initial condition  $\phi_\Lambda$
such that \eqref{aim} holds for all $t>0$.
We will use the topological degree to reach this goal.
 
\vskip20pt

\section{Proof of Theorem \ref{T3.1}}\label{S5}

In this section, we prove Theorem \ref{T3.1}.
The idea of the proof was  outlined in \S \ref{S3.3}.

\vskip5pt

As before, we denote ${\bf p}=(p_1, \cdots, p_k)\in\Om^k$, ${\bf q} = (q_1, \cdots, q_l)\in (\pom)^l$,
 ${\bf a} = (a_1, \cdots , a_{k})$ and ${\bf b} = (b_1, \cdots,  b_{l})$.
Denote for brevity
\beq \label{Lam1}
 {\begin{split}
 \Lambda &=  ({\bf p, q, a, b})=(p_1, \cdots, p_k, q_1, \cdots, q_l, 
              a_1, \cdots, a_k, b_1, \cdots, b_{l} ),\\
 \Lambda' &=  ({\bf p, q}, {\bf{a}}', {\bf{b}}')=(p_1, \cdots, p_k, q_1, \cdots, q_l, 
              a'_1, \cdots, a'_k, b'_1, \cdots, b'_{l} ),
              \end{split}}
\eeq
where $p_i\in\Om, q_j\in\pom, a_i, b_j\in\R^{1, +}$, $1\leq i\leq k$, $1\leq j\leq l$, 
$$a'_i=\frac{a_i}{\bar a}, \ \ b'_j=\frac{b_j}{\bar a}, \ \  \ i=1, \cdots, k, \ j=1, \cdots, l, $$
and $\bar a$ is the average of $a_1, \cdots, a_k, b_1, \cdots, b_l$.
Define
\beq \label{G}
{\begin{split}
G=  \big\{\Lambda' \ | \ & \text{$\Lambda'$ satisfies the following properties}: \\
      & \text{{\bf p} and {\bf q} satisfy the ${\delta}$-apart condition;}\\
      & \text{$|a'_i - 1|<\hat\delta$,\ $|b'_j -1| <\hat\delta$,  \ $\forall\ 1\le i\le k, \ 1\le j\le l$ \big\} . }
      \end{split}} \eeq
By definition, $G\subset \Om^k\times (\pom)^l\times \R^{k+l-1}$,
which is a manifold of dimension 
$$nk+(n-1)l+k+l-1=(n+1)k+nl -1.$$
 
Two special cases: $k=0$ and $l=0$.
\begin{itemize}
\item [(i)] If $l=0$, then the peak solution has interior peaks only, and
$${\begin{split}
G &=  \big\{\Lambda' =({\bf p, \, a'}) \ | \ \text{{\bf p} satisfy the ${\delta}$-apart condition;}\ 
  {  \text{$| a'_i - 1|<\hat\delta$,\ $\forall\ 1\le i\le k$} } \}.
      \end{split}} $$
\item [(ii)] If $k=0$, then the peak solution has boundary peaks only, and
$${\begin{split}
G &=  \big\{\Lambda' =({\bf q, \, b'}) \ | \ \text{{\bf q} satisfy the ${\delta}$-apart condition;}\ 
 { \text{$|b'_j -1|<\hat\delta$,\ $\forall\ 1\le j\le l$}  } \}.
      \end{split}} $$
\end{itemize}
Note that when $k=1$ and $l=0$, $G=\Om_{\delta}$ by the ${\delta}$-apart condition.

Let
$$\phi_\Lambda={\Small\text{$\sum_{i=1}^{k}$}} a_iw_{p_i, \eps}  
+ {\Small\text{$\sum_{j=1}^{l}$}} b_jw_{q_j, \eps} ,$$
where $w_{p, \eps}=w(\frac{x-p}{\eps})$. 
Let $u_\Lambda (\cdot, t)$ be the solution to parabolic p-Laplace equation  \eqref{PE}-\eqref{PEb} 
with initial condition $\phi_\Lambda$.
At the initial time $t=0$, we assume that $\bar a=1$.

The proof of Theorem \ref{T4.1} consists of the following steps. 

\begin{itemize}
\item [\it Step 1.] 
There is a sufficiently small constant $\sigma>0$ such that 
$$\I_{s,\eta}(u_\Lambda(\cdot, t))\leq S_s^*-\sigma$$ 
when $t=0$, for all boundary point $\Lambda\in\p G$. This is verified by Lemmas \ref{L3.4} - \ref{L3.5}, 
together with the concavity property in \S \ref{S3.5}. 
Although the arguments in Section \ref{S3} were carried out for the functional $\I$, 
the same conclusion holds for $\I_s$ and $\I_{s, \eta}$ since $s$ is chosen sufficiently small
and $\eta$ is appropriately selected as in \eqref{eta1}.
               
\vskip5pt

\item [\it Step 2.] 
For any $\Lambda\in G$ and $t>0$, if $\I_{s, \eta}(u_\Lambda(\cdot, t))\ge S_s^*-\sigma$, then $u_\Lambda(\cdot, t)$
is a peak function with interior peaks ${\bf p}={\bf p}_{\Lambda, t}$ and boundary peaks ${\bf q}={\bf q}_{\Lambda, t}$ 
satisfying the ${\delta}$-apart condition, 
and the coefficients ${\bf a}={\bf a}_{\Lambda, t}, {\bf b}={\bf b}_{\Lambda, t}$ satisfying the restrictions 
in Definition \ref{D2.1}  (see Corollary \ref{C4.2}).
\vskip5pt

Recall that if $\I_{s,\eta}(u_\Lambda (\cdot, t))\le S_s^*-\sigma$, then the flow is frozen as in  \S \ref{S4.3}. 
%%then we freeze the flow as in \S \ref{S4.3}. 
 
 \vskip5pt
{%\color{blue}
\item [\it Step 3.]
Let $\Lambda_t=({\bf p}_{\Lambda, t}, {\bf q}_{\Lambda, t}, {\bf a}_{\Lambda, t}, {\bf b}_{\Lambda, t})$
and $\Lambda'_t=({\bf p}_{\Lambda, t}, {\bf q}_{\Lambda, t}, {\bf a'}_{\Lambda, t}, {\bf b'}_{\Lambda, t})$
be as in \eqref{Lamt}.
Note that $\Lambda=\Lambda'$ at $t=0$,  since $\bar a=1$ at $t=0$.

At any time $\tau >0$, if $\Lambda'_{\tau }\in \p G$, then $\I_{s,\eta}(u_{\Lambda}(\cdot, \tau))\le S_s^*-\sigma$ 
(see Corollary \ref{C4.3}). 
Therefore, $u_\Lambda(\cdot, t)$ is frozen after $t>\tau$. 

As a result, we obtain a mapping
$$T_t:\ \Lambda\in G \to\Lambda'_t\in G$$ 
for all time $t\ge 0$, such that $T_0=Id$.  

By Step 1, the mapping $T_t=Id$ on $\p G$.
By Step 2 and Step 3, $T_t(\Lambda)\in G$  for all  $\Lambda\in G$. 
Moreover, by the regularity to the parabolic equation \eqref{PE}-\eqref{PEb},
the mapping $T_t$ is continuous. 
Therefore, by the Brouwer degree theory,
$$T_t(G)=G \ \ \ \mbox{for any} \ t>0.$$
}

%\newpage

\item[\it Step 4.] 
{%\color{blue}
For any given $\tau>0$, since $T_\tau(G)=G$ and
$u_{\Lambda}(\cdot, \tau)$ depends continuously on $\tau$,
by the definition of $S_s^*$,  we have
$$\sup\{ \I_{s, \eta}(u_\Lambda(\cdot, \tau))\ |\ \Lambda\in G\} \ge S^*_s. $$
Hence there is a point $\Lambda=\Lambda_\tau\in G$ such that 
the solution $u_\Lambda(\cdot, t)=u_{\Lambda_\tau} (\cdot, t)$ satisfies 
$\I_{s,\eta}(u_{\Lambda_\tau}(\cdot, \tau))\ge S_s^*$. 
Since the flow \eqref{PE}-\eqref{PEb} is a descent gradient flow for the functional $\I_{s,\eta}$, we see that
$$\I_{s,\eta}(u_{\Lambda_\tau}(\cdot, \tau')) \ge S_s^*\ \ \text{for all}\  \tau'\in [0, \tau ). $$
By passing to a subsequence, we have $\Lambda_\tau\to \Lambda_0\in G$ as $\tau\to\infty$.
By Step 3, $\Lambda_0\not\in\p G$. Moreover, 
$$\I_{s,\eta} (u_{\Lambda_0}(\cdot, t)) \ge  S_s^* \ \ \mbox{for any} \ t>0.$$
This is exactly the situation \eqref{aim}. Therefore, the conditions of \S\ref{S4.5} are satisfied, and the desired conclusion follows.

}
\end{itemize}
  
This completes the proof of Theorem \ref{T3.1}.

\noindent{\bf Remarks}\\
(i) 
The mapping $T_t$ is defined on $G$, 
which is a subset of $\Om^k\times (\pom)^l\times \R^{k+l-1}$.
We can extend the mapping $T_t$ to all points 
${\bf p}=(p_1, \cdots, p_k)\in\Om^k$ and ${\bf q}=(q_1, \cdots, q_l)\in (\pom)^l$,
by simply letting $T_t=Id$ if {\bf p} and {\bf q} do not satisfy the $\delta$-apart condition.
Then $T_t$ is continuous,
and well-defined for all $(p_1, \cdots, p_k)\in\Om^k$ and
$(q_1, \cdots, q_l)\in (\pom)^l$.\\[5pt]
(ii)
There is a decomposition for the mapping $T_t$ as follows:
$$T_t=T_{1,t}\oplus T_{2, t}\oplus T_{3, t},$$
where 
$T_{1, t}$ is the map from $\Om^k\to \Om^k$,
$T_{2, t}$ is the map from $(\pom)^l\to (\pom)^l$,
 and $T_{3, t}$ is the map from $\bar M\to \bar M$ with
 $$
\bar{M} =  \big\{(a_1, \cdots, a_k, b_1,\cdots, b_j) |  \ 
{ \big| a'_i -1\big|<\hat{\delta}, \ \big| b'_j -1\big|<\hat{\delta}, \ \forall \ 1 \leq i\leq k, 1\leq j\leq l\big\}. }
 $$
 %($M$ is given in \S3.1). \\[5pt]
(iii) In the case $k=0, l=1$, 
we always have $\I(u)>S^*$, 
namely $\I(u_p(\cdot, t))<S^*-\sigma$ never happen.
In this case, a gradient flow with any initial condition will converges to a solution.

\end{document}